\documentclass{amsart}
\usepackage{amssymb,latexsym,graphics,amsfonts}

\swapnumbers
 \newtheorem{theorem}{Theorem}[section]
\newtheorem{corollary}[theorem]{Corollary}

\newtheorem{proposition}[theorem]{Proposition}
\theoremstyle{definition}
\newtheorem{definition}[theorem]{Definition}
\theoremstyle{definition}
\newtheorem{example}[theorem]{Example}
\newtheorem{remark}[theorem]{Remark}
\numberwithin{equation}{section}
 \numberwithin{equation}{subsection}

\begin{document}

\title{Generalized notions of module character amenability}

\author[A. Bodaghi]{A. Bodaghi}
\address{Department of Mathematics, Garmsar Branch, Islamic Azad University,  Garmsar,
 Iran}
\email{abasalt.bodaghi@gmail.com}
\author[H. Ebrahimi]{H. Ebrahimi}
\address{Department of Mathematics, University of Isfahan, P.O.BOX 81764-73441, Isfahan, Iran}
\email{ebrahimi89phd@sci.ui.ac.ir}
\author[M. L. Bami]{M. Lashkarizadeh Bami}
\address{Department of Mathematics, University of Isfahan, P.O.BOX 81764-73441, Isfahan, Iran}
\email{lashkari@sci.ui.ac.ir, lashkari@math.ui.ac.ir}



\dedicatory{}
\smallskip

\begin{abstract}
In this paper, we study the hereditary properties
of module $(\phi,\varphi)$-amenability on Banach algebras. We also define the concept of module character contractibility for Banach algebras and obtain characterizations of module character contractible Banach algebras in terms of the existence of module $(\phi,\varphi)$-diagonals.  We introduce module approximately character amenable Banach algebras. Finally,  for every inverse semigroup $S$ with subsemigroup $E$ of idempotents, we find necessary and sufficient conditions for the  $\ell^1(S)$ and its second dual to be module approximate character amenable (as a $\ell^1(E)$-module).
\end{abstract}

\maketitle

\section{Introduction}

M. Amini \cite{am1} introduced the notion of module amenability for a class of Banach algebras which could be considered as a generalization of the Johnson's
amenability \cite{joh}.  He showed that for an inverse semigroup $S$ with the set of idempotents $E$, the semigroup algebra $\ell^1(S)$ is module amenable, as a Banach module over $\ell^1(E)$, if and only if $S$ is amenable. This notion is modified in \cite{boda}, using module homomorphisms
between Banach algebras. Motivated by $\phi$-amenability and character amenability which are studied in \cite{kan1}, \cite{kan2} and \cite{mon}, Bodaghi and Amini \cite{bod} introduced the concept of module $(\phi,\varphi)$-amenability for Banach algebras  and investigated a module character amenable
Banach algebra. They showed that such Banach algebras posses module character virtual (approximate) diagonals. Also, in \cite{bod}, the authors have characterized the module $(\phi,\varphi)$-amenability of a Banach algebra $\mathcal A$ through vanishing of the first Hochschild module cohomology group $\mathcal H^1_{\mathfrak A}(\mathcal A, X^*)$ for certain Banach $\mathcal A$-bimodules $X$. In \cite{pa}, Aghababa and Bodaghi introduced the concepts of module (uniform) approximate amenability and contractibility for Banach algebras (for the classical cases of such notions, see \cite{gh1} and \cite{gh2}). They proved that $\ell^1(S)$ is module approximately amenable (contractible) if and only if it is module uniformly approximately amenable if and only if $S$ is amenable. Also, they showed that the module (uniform) approximate amenability of $\ell^1(S)^{**}$ is equivalent to the finitness of a maximal group homomorphic image of $S$. Furthermore, in \cite{pa}, the authors provided some examples of Banach algebras that are module approximately amenable but not  approximately amenable (for more information about other notions of module amenability, see \cite{ame}, \cite{bodaghi},  \cite{bo3}, \cite{bab}, \cite{lva} and \cite{pou}).

In this paper,  we continue the investigation of the module character amenability which is begun in \cite{bod} and study the hereditary properties
of module $(\phi,\varphi)$-amenability  on Banach algebras. We also give a characterizations of module
$(\phi,\varphi)$-amenability in terms of Hahn-Banach type extension property (module version). As an example, we indicate a module character amenable Banach algebra which is not character amenable. In section 4, we define module character contractible Banach algebras. The main result of this
section asserts that module character contractibility is equivalent to
the existence of a module diagonal and having a left identity. In the last section, we define the concept of module approximate $(\phi,\varphi)$-amenability on Banach algebras and characterize the structure of such Banach algebras. We show that for
an inverse semigroup $S$ with a set of idempotents $E$,
$\ell^1(S)$ and its second dual are module approximate character amenable (with respect to
the special actions).

\section{Preliminaries and Notations}

Let ${\mathcal A}$ and ${\mathfrak A}$ be Banach algebras such
that ${\mathcal A}$ is a Banach ${\mathfrak A}$-bimodule with
compatible actions, that is

\begin{align}\label{e1} \alpha\cdot(ab)=(\alpha\cdot a)b,
\,\,(ab)\cdot\alpha=a(b\cdot\alpha) \hspace{0.3cm}(a,b \in
{\mathcal A},\alpha\in {\mathfrak A}). 
\end{align}
Let ${X}$ be a Banach ${\mathcal A}$-bimodule and a
 Banach ${\mathfrak A}$-bimodule with compatible actions, that is

$\hspace{1.5cm} \alpha\cdot(a\cdot x)=(\alpha\cdot a)\cdot x,
\,\,a\cdot(\alpha\cdot x)=(a\cdot\alpha)\cdot x, \,\,(\alpha\cdot
x)\cdot a=\alpha\cdot(x\cdot a)$\\for all $a \in{\mathcal
A},\alpha\in {\mathfrak A},x\in{X}$ and similarly for
the right and two-sided actions. Then we say that ${X}$ is a
Banach ${\mathcal A}$-${\mathfrak A}$-module. If moreover
$\alpha\cdot x=x\cdot\alpha $ for all $\alpha\in {\mathfrak A}$
and $x\in X$, then $ X $ is called a {\it commutative}
${\mathcal A}$-${\mathfrak A}$-module. Note that when ${\mathcal
A}$ acts on itself by algebra multiplication, it is not in
general a Banach ${\mathcal A}$-${\mathfrak A}$-module. Indeed, if $\mathcal A$ is a commutative $\mathfrak A$-module and
acts on itself by multiplication from both sides, then it is also
a Banach ${\mathcal A}$-${\mathfrak A}$-module.

A bounded map $D: \mathcal A \longrightarrow  X $ is called a
$\mathfrak A$-{\it module derivation} if it is $\mathfrak A$-bimodule homomorphism and
$$D(a\pm b)=D(a)\pm D(b),\hspace{0.2cm}D(ab)=D(a)\cdot b+a\cdot D(b), \qquad (a,b \in \mathcal A).$$
 Note that $D$ is not necessarily
linear, but its boundedness (defined as the existence of $M>0$
such that $\|D(a)\|\leq M\|a\|$, for all $a\in \mathcal A$) still
implies its continuity, as it preserves subtraction. There are plenty of known examples of non linear,
additive derivations (see for instance \cite{s})
and some of these are also module derivations (at least for the trivial action). On the other hand,
when $\mathfrak A$ is unital (or even has a bounded approximate identity) then each module
derivation is automatically linear \cite[Proposition 2.1]{am1}.

If $X$ is a
commutative ${\mathcal A}$-${\mathfrak A}$-module, then each $x
\in  X $ defines an inner module derivation as $D_x(a)=a\cdot x-x\cdot a$
for all $a \in{\mathcal A}$. The Banach algebra ${\mathcal A}$ is called
{\it module amenable} (as an ${\mathfrak A}$-module) if for any
commutative Banach ${\mathcal A}$-${\mathfrak A}$-module $ X $,
each ${\mathfrak A}$-module derivation $D:\mathcal A
\longrightarrow  X^*$ is inner \cite{am1}.

\begin{definition}
Let $X$ be a commutative Banach $\mathcal A$-$\mathfrak A$-module, then a module derivation $D:\mathcal A\longrightarrow X$ is approximately inner, if there exists a net $\{x_i\}\subset X$ such that $D(a)=\lim_i(a\cdot x_i-x_i\cdot a)$ for all $a\in\mathcal A$.
\end{definition}
Note that $\{x_i\}$ in the above definition is not necessarily bounded. Let $\mathfrak A$ be a Banach algebra with character space
$\Phi_{\mathfrak A}$ and let $\mathcal A$ be a Banach $\mathfrak
A$-bimodule with compatible actions. Put $\varphi\in\Phi_{\mathfrak
A}$. Consider the linear map $\phi: \mathcal
A\longrightarrow\mathfrak A$ such that
$$\phi(ab)=\phi(a)\phi(b),\quad \phi(a\cdot\alpha)=\phi(\alpha\cdot a)=\varphi(\alpha)\phi(a)\quad(a\in \mathcal A,\,\alpha\in \mathfrak A).$$
We denote the set of all such maps by $\Omega_{\mathcal A}$.
\begin{definition}
Let $\mathcal A$ be a Banach $\mathfrak A$-bimodule and $\varphi\in\Phi_{\mathfrak A}$ and $\phi\in\Omega_{\mathcal A}$. We say the Banach space $X$ is $((\phi,\varphi),\mathcal A$-$\mathfrak A)$-bimodule if left module action $\mathcal A$ on $X$ given $a\cdot x=\phi(a)\cdot x$ and the action $\mathfrak A$ on $X$ given by $\alpha\cdot x=x\cdot \alpha=\varphi(\alpha)x$  for all $a\in\mathcal A,\alpha\in\mathfrak A$ and $x\in X$. Similarly, we say $X$ is $(\mathcal A$-$\mathfrak A,(\phi,\varphi))$-bimodule, if right module action $\mathcal A$ on $X$ given by $x\cdot a=\phi(a)\cdot x$ and action $\mathfrak A$ on $X$ given by $\alpha\cdot x=x\cdot\alpha=\varphi(\alpha)x$.
\end{definition}


\begin{definition}Let $\mathcal A$ be a Banach $\mathfrak A$-bimodule and $\varphi\in\Phi_{\mathfrak A}$ and $\phi\in\Omega_{\mathcal A}$. Then
\item[(i)] { $\mathcal A$ is module $(\phi,\varphi)$-amenable \emph{(}contractible\emph{)}, if every derivation $D:\mathcal A\longrightarrow X^*$ \emph{(}$D:\mathcal A\longrightarrow X$\emph{)} is inner for all $((\varphi,\phi),\mathcal A$-$\mathfrak A)$-bimodule $X$;}
\item[(ii)] {$\mathcal A$ is right \emph{[}left\emph{]} module character amenable (contractible), if every $((\phi,\varphi),\mathcal A$-$\mathfrak A)$-bimodule \emph{[}\emph{(}$\mathcal A$-$\mathfrak A,(\phi,\varphi)$\emph{)}-bimodule\emph{]} $X$, every derivation $D:\mathcal A\longrightarrow X^*$ \emph{(}$D:\mathcal A\longrightarrow X$\emph{)} is inner;}
\item[(iii)] { $\mathcal A$ is module character amenable \emph{(}contractible\emph{)} if it is both left and right module character amenable \emph{(}contractible\emph{)}.}
\end{definition}

\begin{definition}Let $\mathcal A$ be a Banach $\mathfrak A$-bimodule and $\varphi\in\Phi_{\mathfrak A}$ and $\phi\in\Omega_{\mathcal A}$. Then
\item[(i)] {$\mathcal A$ is approximate module $(\phi,\varphi)$-amenable, if every derivation $D:\mathcal A\longrightarrow X^*$ is approximately inner for all $((\varphi,\phi),\mathcal A$-$\mathfrak A)$-bimodule $X$;}
\item[(ii)] { $\mathcal A$ is right \emph{[}left\emph{]} approximate module character amenable, if every $((\phi,\varphi),\mathcal A$-$\mathfrak A)$-bimodule \emph{[}($\mathcal A$-$\mathfrak A,(\phi,\varphi)$)-bimodule\emph{]} $X$, every derivation $D:\mathcal A\longrightarrow X^*$ is approximately inner;}
\item[(iii)] {$\mathcal A$ is approximate module character amenable if it is both left and right approximately module character amenable.}
\end{definition}


\begin{remark}
Any statement about approximate module left character amenability \emph{(}contractibility\emph{)} turns into an analogous statement about approximate module right character amenability \emph{(}contractibility\emph{)} by simple replacing $\mathcal A$  by its opposite algebra $\mathcal A^{\text{op}}$.
\end{remark}

One should remember that if $\mathfrak A=\mathbb{C}$ and $\varphi$ is the identity
map then all of the above definitions coincide with their classical case (see \cite{hum}, \cite{kan1}, \cite{kan2}, \cite{mon} and \cite{ps}).



\section{Module character amenability}

Throughout this section, we assume that $\mathcal A$ is a Banach $\mathfrak{A}$-bimodule with compatible actions (\ref{e1}) and $\varphi\in \Phi_{\mathfrak{A}}, \phi\in\Omega_{\mathcal A}$. The canonical
images of $a\in\mathcal{A}$ and $\mathcal{A}$ in $\mathcal{A}^{**}$
are denoted by $\widehat{a}$ and $\widehat{\mathcal{A}}$,
respectively.
\begin{definition}\emph{\cite{bod}}
A bounded linear functional $m: \mathcal A^*\to \mathbb{C}$ is called
a {\it module $(\phi,\varphi)$-mean} on $\mathcal A^*$ if $m(f\cdot
a)=\varphi\circ\phi(a)m(f)$, $m(f\cdot\alpha)=\varphi(\alpha)m(f)$
and $m(\varphi\circ\phi)=1$ for each $f\in \mathcal A^*, a\in
\mathcal A$ and $\alpha\in \mathfrak A$.
\end{definition}

It is proved in \cite[Theorem 2.1]{bod}  that $\mathcal A$ is module $(\phi,\varphi)$-amenable if and only if there exists a module
$(\phi,\varphi)$-mean on $\mathcal A^*$. Two upcoming theorems which characterize module $(\phi,\varphi)$-amenability of a Banach algebra, are the module versions of \cite[Theorem 1.4]{kan1} and \cite[Theorem 1.2]{kan2}, respectively. The proofs are similar but we give them for the sake of completeness.

\begin{theorem}\label{th1}
Let $\mathcal A$ be a Banach $\mathfrak A$-module. Then the following statements are equivalent:
\begin{enumerate}
\item[(i)] {$\mathcal A$ is module $(\phi,\varphi)$-amenable;}
\item[(ii)] {There exists a bounded net $(a_j)$ in $\mathcal A$ such that
$\|aa_j-\varphi\circ\phi(a)a_j\| \rightarrow 0$, $\|\alpha \cdot a_j-\varphi(\alpha)a_j\| \rightarrow 0$ for all $a\in\mathcal A, \alpha\in\mathfrak A$
and $\varphi\circ\phi(a_j)=1$ for all $j$.}
\end{enumerate}
\end {theorem}
\begin{proof} (i)$\Rightarrow$(ii) Suppose that $\mathcal A$ is module $(\phi,\varphi)$-amenable and thus a $(\phi,\varphi)$-mean $m\in\mathcal A^{**}$ exists. Take a net $(u_j)$ in $\mathcal A$ with the property that $\widehat{u}_j\rightarrow m$ in the $w^*$-topology of $\mathcal{A}^{**}$ in which $\|u_j\|\leq m$ for all $j$. Since $\langle \varphi\circ\phi , u_j\rangle\rightarrow\langle\varphi\circ\phi,m\rangle=1$, after passing to a subnet and replacing $u_j$ by $(\frac{1}{\varphi\circ\phi(u_j)})u_j$ we can assume that $\varphi\circ\phi(u_j)=1$ and $\|u_j\|\leq \|m\|+1$ for all $j$. Consider the product space ${\mathcal A}^{\mathcal A}$ endowed with the product of norm topologies. Then ${\mathcal A}^{\mathcal A}$ is a locally convex topological vector space. Define a linear map $T:\mathcal A\longrightarrow{\mathcal A}^{\mathcal A}$ via $T(x)=(a\cdot x-\varphi\circ\phi(a)x+\alpha\cdot x-\varphi(\alpha)x)_{a\in {\mathcal A}}$, for all  $x\in{\mathcal A},\alpha\in\mathfrak A$. Put $$B=\{x\in{\mathcal A};\,\, \|x\|\leq \|m\|+1 \,\, \text {and} \hspace{.2cm} \varphi\circ\phi(x)=1\}\subseteq\mathcal A.$$
Clearly, $B$ is convex and so $T(B)$ is a convex subset of ${\mathcal A}^{\mathcal A}$. Since $\varphi\circ\phi(u_j)=1$ for all $j$, the zero element of ${\mathcal A}^{\mathcal A}$ is contained in the closure of $T(B)$ with respect to the product of the weak topology. Thus for all $f\in\mathcal A^*$, we have
\begin{align*}
\langle f, au_j-\varphi\circ\phi(a)u_j+\alpha\cdot u_j-\varphi(\alpha)u_j\rangle&=\langle f,au_j\rangle-\varphi\circ\phi(a)\langle f,u_j\rangle+\langle f,\alpha\cdot u_j\rangle-\varphi(\alpha)\langle f,u_j\rangle\\
&=\langle f\cdot a,u_j\rangle-\varphi\circ\phi(a)\langle f,u_j\rangle+\langle f\cdot\alpha,u_j\rangle-\varphi(\alpha)\langle f,u_j\rangle\\
&\rightarrow\langle m,f\cdot a\rangle-\varphi\circ\phi(a)\langle m,f\rangle+\langle m,f\cdot\alpha\rangle-\varphi(\alpha)\langle m,f\rangle\\
&=0
\end{align*}
This product of weak topologies coincide with the weak topology on ${\mathcal A}^{\mathcal A}$; \cite[Theorem 4.3]{sc}.
By Mazur's theorem, the weak closure of $T(B)$ equals the closure of $T(B)$ in the norm topology on ${\mathcal A}^{\mathcal A}$. In other words,
$$0\in \overline{T(B)}^{w}=\overline{T(B)}^{\|.\|}.$$
It follows that there exists a bounded net $(a_j)_j$ in $B$ such that $\|T(a_j)_j\|\rightarrow0$ and this means that $\varphi\circ\phi(a_j)=1$ and
$$\|aa_j-\varphi\circ\phi(a)a_j\|\rightarrow 0\hspace{.5 cm}\text{and}\hspace{.5cm}\|\alpha \cdot a_j-\varphi(\alpha)a_j\|\rightarrow 0$$
for all $j$ and all $a\in \mathcal A,\alpha\in\mathfrak A$.

(ii)$\Rightarrow$(i) Assume that a net $(a_j)_j$ exists. Let $m$ be $w^*$-cluster point net $(a_j)_j$ in $\mathcal{A}^{**}$. Then
$$\langle m,\varphi\circ\phi\rangle=\lim_j\langle \varphi\circ\phi,a_j\rangle=1.$$
On the other hand,
\begin{align*}
\langle m,f\cdot a\rangle & =\lim_j\langle f\cdot a, a_j\rangle=\lim_j\langle f,aa_j\rangle\\
&=\lim_j\langle f,aa_j-\varphi\circ\phi(a)a_j\rangle+\lim_j\langle f,\varphi\circ\phi(a)a_j\rangle\\
&=\varphi\circ\phi(a)\langle m, f\rangle
\end{align*}
and
\begin{align*}
\langle m,f\cdot\alpha\rangle&=\lim_j\langle f\cdot\alpha,a_j\rangle=\lim_j\langle f,\alpha\cdot a_j\rangle\\
&=\lim_j\langle f,\alpha\cdot a_j-\varphi(\alpha)a_j\rangle+\lim_j\langle f,\varphi(\alpha)a_j\rangle\\
&=\varphi(\alpha)\lim_j\langle f,a_j\rangle=\varphi(\alpha)\langle m,f\rangle
\end{align*}
for all $a\in\mathcal A$ and $f\in \mathcal {A}^*$. This finishes the proof.
\end{proof}

\begin{theorem}\label{th2}
Let $\mathcal A$ be a Banach $\mathfrak A$-module. Then the following statements are equivalent:
\begin{enumerate}
\item[(i)] {$\mathcal A$ is module $(\phi,\varphi)$-amenable;}
\item[(ii)] {If $X$ is any Banach $\mathcal A$-$\mathfrak A$-module and $Y$ is any Banach $\mathcal A$-$\mathfrak A$-submodule of $X$ and $g\in Y^*$ such that the left action of $\mathfrak A$ and $\mathcal A$ on $g$ are $\alpha\cdot g=\varphi(\alpha)g$ and $a\cdot g=\varphi\circ\phi(a)g$, respectively, for all $a\in\mathcal A, \alpha\in\mathfrak A$, then $g$ extends to some $f\in X^*$ such that $a\cdot f=(\varphi\circ\phi)(a)f,\alpha\cdot f=\varphi(\alpha)f$ for all $a\in\mathcal A,\alpha\in\mathfrak A$.}
\end{enumerate}
\end {theorem}
\begin{proof}
(i)$\Rightarrow$(ii) Let $\tilde{g}\in X^*$ such that $\tilde{g}$ extends $g$ and $\|\tilde{g}\|=\|g\|$. Take $a\in \mathcal A$ satisfies $\varphi\circ\phi(a)=1$. Then $a\cdot\tilde{g}$ also extends $g$. Since $\mathcal A$ is module $(\phi,\varphi)$-amenable, by the pervious theorem there exists a net $(u_j)_j$ in $\mathcal A$ such that, for all $j$, we get $\varphi\circ\phi(u_j)=1$ and $\|u_j\|\leq C$ for some constant $C>0$ and $$\|au_j-\varphi\circ\phi(a)u_j\|\rightarrow0\hspace{.5cm},\hspace{.5cm}\|\alpha\cdot u_j-\varphi(\alpha)u_j\|\longrightarrow0$$ for all $a\in\mathcal A,\alpha\in\mathfrak A$; see Theorem \ref{th1}.
Then, $u_j\cdot\tilde{g}$ extends $g$ and we may assume that $\|u_j\cdot\tilde{g}\|\leq C\|g\|+1$ for all $j.$
After passing to a subnet, one can also assume that $u_j\cdot\tilde{g}\rightarrow f$ in the $w^*$-topology for some $f\in X^*.$
Obviously, $f$ extends g because for any $y\in Y$, we have
$$\langle f,y \rangle=\lim_j\langle u_j\cdot\tilde{g},y\rangle=\langle g,y\rangle$$
Taking $w^*$-limits for all $a\in \mathcal A,\alpha\in\mathfrak A$, we deduce that
\begin{align*}
a\cdot f&=\lim_j a\cdot(u_j\cdot\tilde{g})=\lim_j(a\cdot u_j)\cdot\tilde{g}\\
&=\lim_j\{(au_j)\cdot\tilde{g}-\varphi\circ\phi(a)u_j\cdot\tilde{g}+\varphi\circ\phi(a)u_j\cdot\tilde{g}\}\\
&=\lim_j\{au_j-\varphi\circ\phi(a)u_j\}\cdot\tilde{g}+\lim_j \varphi\circ\phi(a)u_j\cdot\tilde{g}\\
&=\varphi\circ\phi(a)f.
\end{align*}
Also,
\begin{align*}
\alpha\cdot f&=\lim_j \alpha\cdot(u_j\cdot\tilde{g})=\lim_j(\alpha\cdot u_j)\cdot\tilde{g}\\
&=\lim_j\{(\alpha u_j)\cdot\tilde{g}-\varphi(\alpha)u_j\cdot\tilde{g}+\varphi(\alpha)u_j\cdot\tilde{g}\}\\
&=\lim_j\{\alpha u_j-\varphi(\alpha)u_j\}\cdot\tilde{g}+\lim_j \varphi(\alpha)u_j\cdot\tilde{g}\\
&=\varphi(\alpha)f.
\end{align*}
for all $f\in X^*$ and all $\alpha\in\mathfrak A$,Therefore (ii) holds.

(ii)$\Rightarrow$(i)  Take $X=\mathcal A^{*}$ and $Y=\mathbb{C}(\varphi\circ\phi)$.
 Let $n\in Y^*$ be defined by $\langle n,\varphi\circ\phi\rangle=1.$ Then, the left action of $\mathcal A$ on $n$ will be by $a\cdot n=\varphi\circ\phi(a)n$ and $\alpha\cdot n=\varphi(\alpha)n$. By assumption, there exists $m\in \mathcal A$ such that $m|_Y=n$ and
 $a\cdot m=\varphi\circ\phi(a)m,\alpha\cdot m=\varphi(\alpha)m$ for all $a\in\mathcal A,\alpha\in\mathfrak A$. We have
 $$\langle m,\varphi\circ\phi\rangle=\langle n,\varphi\circ\phi\rangle=1,$$
 $$\langle m, f\cdot a\rangle=\langle a\cdot m,f\rangle=\langle\varphi\circ\phi(a)m, f\rangle=\varphi\circ\phi(a)\langle m,f\rangle,$$
 and
 $$\langle m,f\cdot\alpha\rangle=\langle\alpha\cdot m,f\rangle=\langle\varphi(\alpha)m,f\rangle=\varphi(\alpha)\langle m,f\rangle$$
for all $\alpha\in\mathfrak A$ and all $a\in\mathcal A$. Therefore, $m$ is a $(\phi,\varphi)$-mean on $\mathcal A^{**}$.
\end{proof}


Let $\mathfrak{F}\in \mathfrak A^{**}$ and $G\in
\mathcal A^{**}$. Take the nets $(\alpha_j)\subset \mathfrak A$ and
$(a_k)\subset \mathcal A$ such that $\widehat{\alpha}_j
\stackrel{w^*}{\longrightarrow} \mathfrak{F}$ and $\widehat{a}_k
\stackrel{w^*}{\longrightarrow} G$. We consider the module actions
$\mathfrak A^{**}$ on $\mathcal A^{**}$ as follows:

$\hspace{1.5cm}\mathfrak{F}\cdot G = w^*-\lim_j w^*-\lim_k
\alpha_j\cdot a_k, \quad G\cdot\mathfrak{F}=w^*-\lim_k w^*-\lim_j
a_k\cdot\alpha_j.$

Let $\phi\in \Omega_{\mathcal A}$ and $\varphi\in\Phi_{\mathfrak
A}$. If $\phi^{**}$ and
$\varphi^{**}$ are the double conjugate of $\phi$
and $\varphi$, respectively, then $\phi^{**}\in
\Omega_{\mathcal A^{**}}$ and $\varphi^{**}\in\Phi_{\mathfrak
A^{**}}$. We denote by $\square$ the first Arens product on $\mathcal{A}^{**}$, the second dual of $\mathcal{A}$.


\begin{proposition}Let $\mathcal A$ be Banach ${\mathfrak A}$-module and let
$\phi\in \Omega_{\mathcal A}$ and $\varphi\in\Phi_{\mathfrak A}$.
Then $\mathcal A$ is module $(\phi,\varphi)$-amenable if and only
if $\mathcal A^{**}$ is module
$(\phi^{**},\varphi^{**})$-amenable.
\end{proposition}
\begin{proof}Let $\mathcal A$ be module $(\phi,\varphi)$-amenable
and $m$ be a $(\phi,\varphi)$-mean in $\mathcal A^{**}$. For each
$F\in \mathcal A^{**}$ and $\Psi\in A^{***}$, take bounded nets
$(a_j)\in \mathcal A$ and $(f_k)\in \mathcal A^{*}$ with $\widehat{a}_j \stackrel{w^*}{\longrightarrow}F$ and $ \widehat{f}_k
\stackrel{w^*}{\longrightarrow}\Psi$. We consider $m$ as an
element $\widehat{m}$ of $\mathcal A^{****}$. Hence,
$\widehat{m}(\varphi^{**}\circ\phi^{**})=m(\varphi\circ\phi)=1$
and
\begin{align*}
\langle \widehat{m}, \Psi\cdot F\rangle
&=\langle \Psi, F\square m\rangle=\lim_k\langle F\square m,f_k\rangle\\
&=\lim_k\langle F,m\cdot f_k\rangle=\lim_k\lim_j\langle m\cdot
f_k, a_j\rangle\\
&=\lim_k\lim_j\langle m,f_k\cdot
a_j\rangle=\lim_k\lim_j\varphi\circ\phi(a_j)\langle m,f_k\rangle\\
&=\lim_j(\varphi\circ\phi)(a_j)\lim_k\langle
m,f_k\rangle=(\varphi^{**}\circ\phi^{**})(F)\langle
\widehat{m},\Psi\rangle.
\end{align*}
Now, for any $\mathfrak{F}\in \mathfrak A^{**}$ and $\Psi\in
A^{***}$, similar to the above computations, we can show that
$\langle \widehat{m}, \Psi\cdot \mathfrak
F\rangle=\varphi^{**}(\mathfrak F)\langle
\widehat{m},\Psi\rangle$. Thus $\mathcal A^{**}$ is module
$(\phi^{**},\varphi^{**})$-amenable.

For the converse, assume that $M$ is a
$(\phi^{**},\varphi^{**})$-mean on $\mathcal A^{***}$
satisfying $M(\varphi\circ\phi)=1, \langle M, \Psi\cdot
F\rangle=(\varphi^{**}\circ\phi^{**})(F)\langle
M,\Psi\rangle$ and $\langle M, \Psi\cdot \mathfrak
F\rangle=\varphi^{**}(\mathfrak F)\langle M,\Psi\rangle$ for
all $F\in \mathcal A^{**},\Psi\in A^{***}$ and $\mathfrak{F}\in
\mathfrak A^{**}$. Then the restriction of $M$ to $\mathcal A^{*}$
is a $(\phi,\varphi)$-mean on $\mathcal A^{*}$.
\end{proof}


Let $\mathcal A$ be a Banach $\mathfrak{A}$-bimodule with compatible actions (\ref{e1}) and $\varphi\in \Phi_{\mathfrak{A}}, \phi\in\Omega_{\mathcal A}$. Similar to the classical case, by applying $w^*$-continuity, it is easily verfied that an element $m\in \mathcal A^{**}$ is a  module $(\phi,\varphi)$-mean for $\mathcal A$  if and only if for all $n\in\mathcal A^{**}$ and $\alpha\in \mathfrak A$, we have
$$nm=(\varphi^{**}\circ\phi^{**})(n)m\,\, \text{and}\,\, \alpha\cdot m=\varphi(\alpha)m.$$
Sometimes a module $(\phi,\varphi)$-mean can lies in $\mathcal A$. However, this motivates us to introduce a new notion of module $(\phi,\varphi)$-mean. We say a $\mathfrak{A}$-bimodule $\mathcal A$ has a module $(\phi,\varphi)$-mean $m$ in itself if $am=\varphi\circ\phi(a)m\,\, \text{and}\,\, \alpha\cdot m=\varphi(\alpha)m$ for all $a\in\mathcal A$ and $\alpha\in \mathfrak A$. Note that a Banach $\mathfrak A$-module $\mathcal A$ has a $(\phi,\varphi)$-mean if and only if $\mathcal A^{**}$ has a $(\phi^{**},\varphi^{**})$-mean which lies in $\mathcal A^{**}$.

Now, let $\mathcal I$ and $\mathfrak I$ be ideals in $\mathcal A$ and $\mathfrak A$, respectively. If $\mathcal I$ is an $\mathfrak A$-submodule of $\mathcal A$ such that $\mathcal I\subseteq \text{ker}\phi$, then $\mathcal A/\mathcal I$ be a Banach $\mathfrak{A}/\mathfrak I$-bimodule with compatible actions (\ref{e1}) and the mapping $\tilde{\phi}:\mathcal A/\mathcal I\longrightarrow\mathfrak{A}/\mathfrak I; a+\mathcal I\mapsto \phi(a)+\mathfrak I$ is well-defined.


\begin{proposition}\label{pro1} With the above hypotheses, let $\mathfrak I\subseteq \emph{ker}\varphi$ and $\tilde{\varphi}:\mathfrak A/\mathfrak I\longrightarrow\mathbb{C}$ be the homomorphism induced by $\varphi$. If $\mathcal I$ has a right identity and $\mathcal A/\mathcal I$ has a module $(\tilde{\varphi},\tilde{\phi})$-mean in $\mathcal A/\mathcal I$, then $\mathcal A$ has a module $(\varphi,\phi)$-mean in $\mathcal A$.
\end{proposition}
\begin{proof}
Let $\mathcal P:\mathcal A\longrightarrow\mathcal A/\mathcal I$ and $\mathfrak P:\mathfrak A\longrightarrow\mathfrak A/\mathfrak I$ be the canonical projections. Assume that $e$ is a right identity for $\mathcal I$ and $m\in\mathcal A$  such that $\mathcal P(m)$ is module $(\tilde{\varphi},\tilde{\phi})$-mean for $\mathcal A/\mathcal I$. Clearly, $\mathcal P(e)=0$. Thus for any $a\in\mathcal A$, we get

\begin{align*}
\mathcal P(a)\mathcal P(m-me)&=\mathcal P(a)\mathcal P(m)=\tilde{\varphi}\circ\tilde{\phi}(\mathcal P(a))\mathcal P(m)\\
&=(\varphi\circ\phi)(a)\mathcal P(m)=(\varphi\circ\phi)(a)\mathcal P(m-me).
\end{align*}
This implies that $a(m-me)-\varphi\circ\phi(m-me)\in \mathcal I$. Since $e$ is a right identity for $\mathcal I$ and $(m-me)e=0$, we have
$$a(m-me)-\varphi\circ\phi(a)(m-me)=a(m-me)e-\varphi\circ\phi(a)(m-me)e=0.$$
So $a(m-me)=(\varphi\circ\phi)(a)(m-me)$. Also, for each $\alpha\in \mathfrak A$, we have
\begin{align*}
\mathcal P(\alpha\cdot(m-me))&=\mathfrak P(\alpha)\mathcal P(m-me)=\mathfrak P(\alpha)\mathcal P(m)\\
&=\tilde{\varphi}(\mathfrak P(\alpha))\mathcal P(m)=\varphi(\alpha)\mathcal P(m)\\
&=\varphi(a)\mathcal P(m-me).
\end{align*}
Hence $\alpha\cdot(m-me)-\varphi(a)(m-me)\in \mathfrak I$. Now, similar the above we can show that $\alpha\cdot(m-me)=\varphi(a)(m-me)$. Furthermore,
\begin{align*}
\langle m-me,\varphi\circ\phi\rangle&=\langle m-me,\tilde{\varphi}\circ\tilde{\phi}\circ\mathcal P\rangle\\
&=\langle\tilde{\varphi}\circ\tilde{\phi},\mathcal P(m-me)\rangle\\
&=\langle \tilde{\varphi}\circ\tilde{\phi},\mathcal P(m)\rangle\\
&=1
\end{align*}
Therefore, $m-me$ is module $(\varphi,\phi)$-mean in $\mathcal A$.
\end{proof}


The following result indicates a necessary and sufficient for being module $(\phi,\varphi)$-amenability of a Banach algebra.
\begin{proposition}
$\emph{ker}(\varphi\circ\phi)$ has a bounded right approximate identity if and only if $\mathcal A$ is module $(\phi,\varphi)$-amenable and has a right approximate idnetity.
\end{proposition}

\begin{proof}Assume that $(u_j)$ is a bounded right
approximate identity for $\ker(\varphi\circ \phi)$. Choose $a_0\in {\mathcal A}$ such that
$\varphi\circ \phi(a_0)=1$. Let $b_j=a_0-a_0u_j$. Then for each $b\in \ker(\varphi\circ \phi)$ we have
$$
\|bb_j\|=\|(ba_0)-(ba_0)u_j\|\rightarrow 0.
$$
It is clear that $(aa_0-\varphi\circ \phi(a)a_0), (\alpha\cdot a_0-\varphi(\alpha)a_0)\in \ker(\varphi\circ \phi)$ for all $a\in{\mathcal A}$ and $\alpha\in
{\mathfrak A}$.
Putting $a_j=a_0b_j$ for all $j$, we have $\varphi\circ\phi(a_j)=1$ and
$$\|aa_j-\varphi\circ\phi(a)a_j\|=\|(aa_0-\varphi\circ\phi(a)a_0)b_j\|\rightarrow 0,$$
$$ \|\alpha\cdot a_j-\phi(\alpha)a_j\|=\|(\alpha\cdot a_0-\phi(\alpha)a_0)b_j\|\rightarrow 0$$
for all $a\in{\mathcal A}$ and $\alpha\in {\mathfrak A}$. Therefore, ${\mathcal A}$ is  module $(\phi,\varphi)$-amenable by Theorem \ref{th1}.
Now, by letting $e_j=u_j-a_0u_j+a_0$ for all $j$,
we get $a-aa_0\in \ker(\varphi\circ \phi)$, and thus
$$
\|ae_j-a\|=\|(a-aa_0)u_j-(a-aa_0)\|\rightarrow 0
$$
for all $a\in{\mathcal A}$. It follows that $(e_j)$
is a right approximate identity for ${\mathcal A}$.

For the converse, it follows firsty from the definition that $\mathcal A$ is $(\phi,\varphi)$-amenable. Now, the proof of \cite[Proposition 2.2]{kan1} can be repeated to obtain the result by interchanging $\varphi$ by $\varphi\circ\phi$.
\end{proof}

Consider the module projective tensor product ${\mathcal A}\widehat
\otimes _{\mathfrak A} {\mathcal A}$ which is isomorphic to the
quotient space $(\mathcal A \widehat \otimes \mathcal A)/{I_\mathcal A}$, where
$I_\mathcal A$ is the closed ideal of the projective tensor product $\mathcal A
\widehat \otimes \mathcal A$ generated by elements of the form $
a\cdot\alpha \otimes b-a \otimes\alpha \cdot b$ for $ \alpha\in
{\mathfrak A},a,b\in{\mathcal A}$ \cite{rie}. Also consider the
closed ideal $J_\mathcal A$ of ${\mathcal A}$ generated by elements of the form
$ (a\cdot\alpha) b-a(\alpha \cdot b)$ for $ \alpha\in {\mathfrak
A},a,b\in{\mathcal A}$. Then $I_\mathcal A$ and $J_\mathcal A$ are ${\mathcal
A}$-submodules and ${\mathfrak A}$-submodules of $\mathcal A
\widehat \otimes \mathcal A$ and $\mathcal A$, respectively, and the
quotients $\mathcal A \widehat \otimes_{\mathfrak A} \mathcal A$ and
$\mathcal A/J_\mathcal A$ are ${\mathcal A}$-modules and ${\mathfrak
A}$-modules. Also $\mathcal A/J_\mathcal A$ is a Banach $\mathcal
A$-${\mathfrak A}$-modules with the canonical action. Throughout,  We shall denote $I_{\mathcal A}$ and
$J_{\mathcal A}$ by $I$ and $J$, respectively, if there is no risk of confusion.

A discrete semigroup $S$ is called
an {\it inverse semigroup} if for each $ s \in S $ there is a unique
element $s^* \in S$ such that $ss^*s=s$ and $s^*ss^*=s^*$. An
element $e \in S$ is called an {\it idempotent} if $e=e^*=e^2$. The
set of idempotents of $S$ is denoted by $E$.

For an inverse semigroup $S$, the
ideal $J_{\ell ^{1}(S)}$ (or $J$)  is the closed linear span of
$\{\delta_{set}-\delta_{st} : s,t \in S,  e \in E \}.$ We consider an equivalence relation on $S$ as follows:
$$s\approx t \Longleftrightarrow \delta_s-\delta_t \in J \qquad (s,t \in
S).$$
In this case the quotient ${S}/{\approx}$ is a
discrete group (see \cite{abe} and \cite{pou}). In fact,
${S}/{\approx}$ is homomorphic to the maximal group homomorphic
image $G_S$ \cite{mn} of $S$ \cite{pou2}. In particular, $S$ is
amenable if and only if ${S}/{\approx}=G_S$ is amenable \cite{dun, mn}. As in \cite[Theorem 3.3]{ra1}, we may observe that $\ell
^{1}(S)/J\cong {\ell ^{1}}(G_S)$. With the notations of the previous
section, $ \ell ^{1}(S)/{J}$ is a commutative $ \ell
^{1}(E)$-bimodule with the following actions
$$\delta_e\cdot\delta_{[s]} = \delta_{[s]}, \,\,\delta_{[s]}\cdot\delta_e = \delta_{[se]} \quad(s \in S,  e \in E),$$
where $[s]$ denotes the equivalence class of $s$ in $G_S$.

In the next example, we bring a module character amenable Banach algebra which is not character amenable.
\begin{example}
Let $G$ be a group with identity $e$, and let $\Gamma$ be a non-empty
set. Then the Brandt inverse semigroup corresponding to $G$ and
$\Gamma$, denoted by $S={\mathcal{M}}(G,\Gamma)$, is the collection of all
$\Gamma \times \Gamma$ matrices $(g)_{ij}$ with $g\in G$ in the
$(i,j)^{\textmd{th}}$ place and $0$ (zero) elsewhere and the $\Gamma
\times \Gamma$ zero matrix $0$. Multiplication in $S$ is given by the
formula
$$(g)_{ij}(h)_{kl}=\left\{\begin{array}{l} (gh)_{il} \qquad {\rm {if}}\
j=k \\ \ \ 0 \qquad \quad \, {\rm{if}}\ j \neq k
\end{array}\right. \qquad\qquad (g,h\in G, \, i,j,k,l\in \Gamma),$$ and
${(g)^*_{ij}}=(g^{-1})_{ji}$ and $0^*=0$. The set of all idempotents
is $E_S=\{(e)_{ii}:i\in \Gamma\}\bigcup \{0\}$.  It is shown in
\emph{\cite{pou}} that $G_S$ is the trivial group. Hence
${\ell ^{1}}(S)$ is module amenable by  \emph{\cite[Theorem 3.1]{am1}} and thus it is module character amenable. But if the index set $\Gamma$ is
infinite, then ${\ell ^{1}}(S)$ is not character amenable \emph{\cite[Corollary 2.7]{ef}}. 
\end{example}


Let $\mathcal N$ be the closed ideal of ${\mathcal
A^{**}}$ generated by $(F \cdot \alpha)\square G-F \square
(\alpha\cdot G)$, for $F,G\in{\mathcal A}^{**}$ and $\alpha\in
{\mathfrak A}$. Then clearly $J\subseteq \mathcal N$. It follows from the proof of \cite[Theorem 3.4]{abeh} that $ \mathcal N \subseteq J^{\perp \perp}$, and thus the map $\lambda : \mathcal A^{**}/\mathcal N\longrightarrow\mathcal A^{**}/J^{\perp \perp}; F+\mathcal N \mapsto F+ J^{\perp \perp}$ is well-defined. It is also a bounded $\mathcal A$-$\mathfrak A$-module homomorphism which is an epimorphism.

In analogy with the classical case we characterize the inverse semigroup that the second dual of its algebras is module character amenable.

\begin{theorem}\label{thj}
Let $S$ be an
inverse semigroup with the set of idempotents $E$.
 Then, $\ell ^{1}(S)^{**}$ is module character amenable \emph{(}as an $\ell
^{1}(E)$-module with trivial left action\emph{)} if and only if the
discrete group $G_S$ is finite.
\end{theorem}

\begin{proof}In light of \cite[Theorem 3.4]{abe}, we need to prove the necessity part of the theorem. Suppose that
$\ell ^{1}(S)^{**}$ is module character amenable. Going back to the case where $\mathcal
A=\ell^1(S)^{**}$ and $\mathfrak A=\ell^1(E)$. We may consider
$J=\mathcal N$ in \cite[Remark 2.5]{bod} applied to $\ell^1(S)^{**}$. Then $\ell^1(S)^{**}/\mathcal N$ is character amenable. Since $\lambda$ is a continuous epimorphism, $\ell ^{1}(S)^{**}/J^{\perp \perp}\cong \ell^1(G_S)^{**}$ is character amenable by \cite[Theorem 2.6]{mon}. Now, it follows from \cite[Theorem 11.17]{dal} that $G_S$ is finite.
\end{proof}


\section{Module character contractibility}

 Let $I=I_{\mathcal A \widehat \otimes \mathcal A}$ and $J=J_\mathcal A$ be the corresponding closed ideals
of $\mathcal A \widehat \otimes \mathcal A$ and $\mathcal A$ be as in section 3, respectively.
Consider the map $\omega :\mathcal A \widehat \otimes \mathcal
A\longrightarrow \mathcal A$ defined by $\omega (a \otimes b)=ab$
and $\widetilde{\omega}: {\mathcal A}\widehat \otimes _{\mathfrak A}
{\mathcal A}\cong(\mathcal A \widehat \otimes \mathcal
A)/{I}\longrightarrow\mathcal A/J$ defined by $ \widetilde{\omega}
(a \otimes b +I)=ab+J$ and extended by linearity and continuity. It
is clear that $ \widetilde{\omega}$ and its double conjugate
$\widetilde{\omega}^{**}: ({\mathcal A}\widehat \otimes _{\mathfrak
A} {\mathcal A})^{**}\longrightarrow\mathcal A^{**}/J^{\perp\perp}$
are both $\mathcal A$-bimodule and ${\mathfrak A}$-bimodule
homomorphisms.


\begin{theorem}\label{th}
Let $\varphi\in\Phi_{\mathfrak A}$ and
$\phi\in\Omega_{\mathcal A}$. Then the following are equivalent:

\begin{enumerate}
\item[(i)] {$\mathcal A$ is module $(\phi,\varphi)$-contractible;}
\item[(ii)] {There exists $m\in \mathcal A$ such that
$\varphi\circ\phi(m)=1$ and $am=\phi(a)\cdot m, \alpha\cdot m=m\cdot\alpha=\varphi(\alpha)m$ for all $a\in
\mathcal A, \alpha\in\mathfrak A$.}
\end{enumerate}
\end{theorem}

\begin{proof}
(i)$\Rightarrow$(ii)  We define the right actions of $\mathcal
A$ on $X=\mathcal A$ by $x\cdot a=\phi(a)\cdot x$
and action of $\mathfrak A$ on $X$ by $\alpha\cdot x=x\cdot\alpha=\varphi(\alpha)x$, and the left
actions of $\mathcal A$ on $X$ is naturally. Take $b\in \mathcal A$ such that
$\varphi\circ\phi(b)=1$. Define a module derivation $D:\mathcal A\longrightarrow
X$ by $D(a)=ab-\phi(a)\cdot b$. obviously, $D(a)$ belongs to
the kernel of $\varphi\circ\phi$. Due to the
$(\phi,\varphi)$-contractibility of $\mathcal A$, there exists a
$n\in\ker(\varphi\circ\phi)$ such that $D(a)=a\cdot n-n\cdot a$. If
we put $m=b-n$, then $\varphi\circ\phi(m)=1$ and
$$am=ab-a\cdot n=ab-(D(a)-n\cdot a)=\phi(a)\cdot b-\phi(a)\cdot n=\phi(a)\cdot m$$
and obviously $\alpha\cdot m=\varphi(\alpha)m$ for all $a\in \mathcal A,\alpha\in\mathfrak A$.

(ii)$\Rightarrow$ (i) Let $\varphi\circ\phi(m)=1$ and
$am=\phi(a)\cdot m,\alpha\cdot m=m\cdot\alpha=\varphi(\alpha)m$ for all $a\in \mathcal A,\alpha\in\mathfrak A$. Suppose that $D:\mathcal A\longrightarrow X$ is module derivation, when the right action $\mathcal A$  on
$X$ is $x\cdot a=\phi(a)\cdot x$ and action of $\mathfrak A$ on $X$ is $\alpha\cdot x=x\cdot\alpha=\varphi(\alpha)x$ for all $\alpha\in \mathfrak A$ and $x\in
X$.Put
$x=D(m)$, then
\begin{align*}
a\cdot x&=D(am)-D(a)\cdot m=D(am)-\phi(m).D(a)\\
&=D(am)-\varphi\circ\phi(m)D(a)=D(\phi(a).m)-D(a)\\
&=\varphi\circ\phi(a)D(m)-D(a).
\end{align*}
Hence, $D(a)=\varphi\circ\phi(a)x-a\cdot x=\phi(a)\cdot x-a\cdot x=x\cdot a-a\cdot x$. This shows that
$\mathcal A$ is module $(\phi,\varphi)$-contractible.
\end{proof}

Let $\mathcal A$ be Banach ${\mathfrak A}$-module, and let $\phi\in \Omega_{\mathcal A}, \varphi\in\Phi_{\mathfrak
A}$. It is obvious that $\phi((a\cdot\alpha)b-a(\alpha\cdot b))=0$, hence $\phi=0$
on $J$ and $\phi$ lifts to $\tilde \phi: \mathcal A/J\to \mathfrak
A$ and thus $\tilde{\phi}\in\Omega_{\mathcal A/J}$. Also, $(\phi\otimes\phi)(a\cdot\alpha \otimes
b-a \otimes\alpha \cdot b)=0$, and so the map
$\overline{\phi\otimes\phi}:{\mathcal A}\widehat \otimes
_{\mathfrak A}{\mathcal A}\cong({\mathcal A}\widehat \otimes
{\mathcal A})/I\longrightarrow\mathfrak A\widehat\otimes\mathfrak A$
defined by $\overline{\phi\otimes\phi}(a \otimes
b+I):=(\phi\otimes\phi)(a \otimes b)$ is well defined.

\begin{definition}Let $\mathcal A$ be Banach ${\mathfrak A}$-module and let
$\phi\in \Omega_{\mathcal A}$ and $\varphi\in\Phi_{\mathfrak A}$. An element $ \widetilde{m} \in {\mathcal A}\widehat \otimes
_{\mathfrak A} {\mathcal A}$ is a module
$(\phi,\varphi)$-diagonal for $\mathcal A$ if
\begin{enumerate}
\item[(i)] {$\widetilde{m}\cdot a=(\varphi\circ\phi)(a)
\widetilde{m},\quad (a\in\mathcal A);$}
\item[(ii)] {$\langle(\varphi\otimes\varphi)\circ(\overline{\phi\otimes\phi}), \widetilde{m}\rangle
=(\varphi\circ\tilde{\phi})\widetilde{\omega}((\widetilde{m}))=1$.}
\end{enumerate}
\end{definition}


Recall that a left Banach $\mathcal A$-module $X$ is called {\it{left essential}} if the linear span of $\mathcal A \cdot X=\{a\cdot x : a\in \mathcal A, \, x\in X\}$ is dense in $X$. Right essential $\mathcal A$-modules and two-sided essential $\mathcal A$-bimodules are defined similarly.

\begin{theorem}Let $\mathcal A$ be Banach left essential ${\mathfrak A}$-module and let
$\phi\in \Omega_{\mathcal A}$ and $\varphi\in\Phi_{\mathfrak A}$. Then

\begin{enumerate}
\item[(i)] {$\mathcal A$ is module $(0,\varphi)$-contractible if and only if $\mathcal A$ has a left identity;}
\item[(ii)] {If $\mathcal A$ is module $(\phi,\varphi)$-contractible, then $\mathcal A$ has a module
$(\phi,\varphi)$-diagonal. The converse is true if  $\mathcal A$ is commutative ${\mathfrak A}$-bimodule;}
\item[(iii)] {$\mathcal A$ is module character contractible if and only if $\mathcal A$ has a left identity and has a module
$(\phi,\varphi)$-diagonal for all $\phi\in \Omega_{\mathcal A}$ and $\varphi\in\Phi_{\mathfrak A}$.}
\end{enumerate}
\end{theorem}

\begin{proof} We follow the argument in \cite[Theorem 6.3]{hum}.

(i) Suppose that $\mathcal A$ is module $(0,\varphi)$-contractible. Let $X_0=\mathcal A\oplus_1\mathcal A$. Consider the module actions $\mathcal A$ and ${\mathfrak A}$ on $X_0$ as follows:
$$a\cdot(b,c)=(0,0),\quad (b,c)\cdot a=(ba,ca)\qquad (a,b,c\in \mathcal A)$$
$$\alpha\cdot(a,b)=(a,b)\cdot \alpha=(\varphi(\alpha)a,\varphi(\alpha)b),\qquad (a,b\in \mathcal A, \alpha\in \mathfrak A).$$
Then, $X_0$ is a Banach $\mathcal A$-${\mathfrak A}$-module with the compatible actions (\ref{e1}). It follows from the assumption that the bounded derivation $D:\mathcal A\longrightarrow X_0$ is inner, and thus $D=D_{(a_0,b_0)}$ for some $(a_0,b_0)\in X_0$. This shows that  $-a_0$ and $-b_0$ are left identities for $\mathcal A$. For the converse, if $a_0$ is a left identity for $\mathcal A$, then for every the bounded derivation $D:\mathcal A\longrightarrow X$, we have $D=D_{D(a_0)}$ for some $a_0\in \mathcal A$.

(ii) Let $\mathcal A$ is module $(\phi,\varphi)$-contractible. We
consider the Banach $\mathcal A$-bimodule ${\mathcal A}\widehat
\otimes {\mathcal A}$ with the module action
$$a\cdot(b\otimes c)=\varphi\circ\phi(a)b\otimes c,\quad (b\otimes
c)\cdot a=b\otimes ca,\qquad(a,b,c\in\mathcal A).$$
Since
$c\cdot(a\cdot\alpha \otimes b-a \otimes\alpha \cdot
b)=(ca\cdot\alpha \otimes b-ca \otimes\alpha \cdot b)\in I$ and
similarly for the right action, $\mathcal A$ acts as a bimodule on
${\mathcal A}\widehat \otimes _{\mathfrak A}{\mathcal
A}$. Also
${\mathcal A}\widehat \otimes _{\mathfrak A}{\mathcal A}$ is a
$\mathfrak A$-bimodule with the following actions:
$$\alpha\cdot(a\otimes b+I)=(a\otimes
b+I)\cdot\alpha=\varphi(\alpha)a\otimes b+I\qquad(a,b\in\mathcal
A,\,\ \alpha\in\mathfrak A).$$
Obviously, the above actions are well-defined. Let $\widetilde{m}_0\in {\mathcal A}\widehat \otimes _{\mathfrak A}{\mathcal
A}$ such that $\langle(\varphi\otimes\varphi)\circ(\overline{\phi\otimes\phi}), \widetilde{m}_0\rangle=1$, and
consider the inner derivation
$$D_{\widetilde{m}_0}:\mathcal A\longrightarrow {\mathcal A}\widehat \otimes _{\mathfrak A}{\mathcal
A}; a\mapsto (\varphi\circ\phi)(a)\cdot\widetilde{m}_0-\widetilde{m}_0\cdot a.$$
Since $\mathcal A$ is a left essential ${\mathfrak A}$-module, it follows from the proof of \cite[Theorem 3.14]{bab} that the map $\varphi\circ\phi$ is $\mathbb C$-linear. A simple computation shows that image of $D_{\widetilde{m}_0}$ is a subset of ker $(\varphi\otimes\varphi)\circ(\overline{\phi\otimes\phi})$. Thus, by the hypothesis there exists $\widetilde{m}_1\in$ ker $(\varphi\otimes\varphi)\circ(\overline{\phi\otimes\phi})$ such that $D_{\widetilde{m}_0}=D_{\widetilde{m}_1}$. Now, it is routine to show that $\widetilde{m}_0-\widetilde{m}_1$ is $(\phi,\varphi)$-diagonal.

Conversely, let $\widetilde{m}$ be a $(\phi,\varphi)$-diagonal for $\mathcal A$. Suppose that $X$ is a Banach $\mathcal A$-module and Banach
${\mathfrak A}$-module such that $a\cdot x=\phi(a)\cdot x$ and
$\alpha\cdot x=x\cdot\alpha=\varphi(\alpha)x$ for all $x\in X,
a\in \mathcal A$ and $\alpha\in \mathfrak A$. Let $D:\mathcal A\longrightarrow X$ be a module derivation. Put $x_0=D(\widetilde{\omega}(\widetilde{m}))$ and $a\in \mathcal A$ such that $(\varphi\circ\phi)(a)=1$. Then
\begin{align*}
D_{x_0}(a)&=a\cdot x_0-x_0\cdot a=a\cdot D(\widetilde{\omega}(\widetilde{m}))-D(\widetilde{\omega}(\widetilde{m}))\cdot a\\
&=\varphi\circ\phi(a)D(\widetilde{\omega}(\widetilde{m}))-[D(\widetilde{\omega}(\widetilde{m})\cdot a)-\widetilde{\omega}(\widetilde{m})\cdot D(a)]\\
&=\widetilde{\omega}(\widetilde{m})\cdot D(a)=(\varphi\circ\tilde{\phi})\widetilde{\omega}((\widetilde{m}))D(a)=D(a).
\end{align*}
Not that in the above statements, we have used this fact that $\widetilde{\omega}(\widetilde{m})\cdot a=\widetilde{\omega}(\widetilde{m}\cdot a)=\widetilde{\omega}(\varphi\circ\phi(a)\widetilde{m})= \varphi\circ\phi(a)\widetilde{\omega}(\widetilde{m})=\widetilde{\omega}(\widetilde{m})$. Since the elements of the form $(\varphi\circ\phi)(a)=1$ span $\mathcal A$, we have $D_{x_0}(a)=D(a)$ for all $a\in \mathcal A$.

(iii) This is a direct consequence of (i) and (ii).
\end{proof}



\section{Module Approximate character amenability}

\begin{definition}
Let $\mathcal A$ be a Banach $\mathfrak A$-bimodule and $\varphi\in\Phi_{\mathfrak A}$ and $\phi\in\Omega_{\mathcal A}$. A net $\{m_j\}_j\subset\mathcal A^{**}$ is called a module \emph{(}uniformly, $w^*$-\emph{)} approximate $(\phi,\varphi)$-mean if $m_j(\varphi\circ\phi)=1$ and for all $a\in\mathcal A$ and $\alpha\in\mathfrak A$
$$a\cdot m_j-\phi(a)\cdot m_j\rightarrow0,\hspace{.5cm}\alpha\cdot m_j-\varphi(\alpha)m_j\rightarrow0$$
\emph{(}uniformly on the unit ball of $\mathcal A$, in the $w^*$-topology of $\mathcal A^{**}$, respectively\emph{)}.
\end{definition}


\begin{theorem}\label{5th1}
Let $\mathcal A$ be a Banach $\mathfrak A$-bimodule and $\varphi\in\Phi_{\mathfrak A}$ and $\phi\in\Omega_{\mathcal A}$. If left action $\mathfrak A$ on $\mathcal A$ is $\alpha\cdot a=\varphi(\alpha)a$, then the following statement are equivalent:
\begin{enumerate}
\item[(i)] {$\mathcal A$ is module \emph{(}uniformly, $w^*$-\emph{)} approximately $(\phi,\varphi)$-amenable;}

\item[(ii)]{There exist a module \emph{(}uniformly, $w^*$-, respectively\emph{)} approximately $(\phi,\varphi)$-mean;}

\item[(iii)]{There exists a net $\{m_j\}_j\subset\mathcal A^{**}$, such that $m_j(\varphi\circ\phi)\rightarrow1$ and for all $a\in\mathcal A,\alpha\in\mathfrak A$
$a\cdot m_j-\phi(a)\cdot m_j\rightarrow 0,\alpha\cdot m_j-\varphi(\alpha)m_j\rightarrow 0$ \emph{(}uniformly on the unit ball of $\mathcal A$, in the $w^*$-topology of $\mathcal A^{**}$, respectively\emph{)};}
\item[(iv)] {Let $K=\emph{ker}(\varphi\circ\phi)$ and $\mathcal A$ act on the $K^{**}$ from the right by action $m\cdot a=\phi(a)\cdot m$ for all $a\in\mathcal A$, $m\in K^{**}$ and taking the left action to the natural one. If $\mathfrak A$ acts on the $K^{**}$ by $\alpha\cdot m=m\cdot\alpha=\varphi(\alpha)m$, for $\alpha\in\mathfrak A$ and $m\in K^{**}$, then any bounded module derivation $D:\mathcal A\longrightarrow K^{**}$ is module \emph{(}uniformly, $w^*$-, respectively\emph{)} approximate inner.}
\end{enumerate}
\end{theorem}
\begin{proof}
(i)$\Rightarrow$(ii)  We firstly define the actions of $\mathcal A$ and $\mathfrak A$ on $X=\mathcal A^*$ through
$$a*f=\phi(a)\cdot f,\hspace{.2cm}f*a=f\cdot a, \hspace{.2cm} \alpha\bullet f=f\bullet\alpha=\varphi(\alpha)f, \hspace{.4cm} (\alpha\in\mathfrak A, a\in\mathcal A, f\in X).$$
Then $X$ is a $((\phi,\varphi),\mathcal A$-$\mathfrak A)$-bimodule and $X^*=\mathcal A^{**}$ is a $(\mathcal A$-$\mathfrak A,(\phi,\varphi))$-bimodule that module action are given by
$$a\ast m=a\cdot m,\hspace{.2cm} m\ast a=\phi(a)m,\hspace{.2cm}\alpha\bullet m=m\bullet\alpha=\varphi(\alpha)m\hspace{.4cm}(\alpha\in\mathfrak A, a\in\mathcal A,m\in X^*).$$
It follows the fact $a\cdot(\varphi\circ\phi)=\phi(a)\cdot(\varphi\circ\phi)$ that $\mathbb{C}\varphi\circ\phi=\{\lambda(\varphi\circ\phi);\lambda\in\mathbb{C}\}$ is a closed submodule of $\mathcal A^*$ and thus $\mathcal A^*/\mathbb{C}\varphi\circ\phi$ is a $((\phi,\varphi),\mathcal A$-$\mathfrak A)$-bimodule, for which the module actions are given by:
$$a\cdot[f]=\phi(a)\cdot[f],\hspace{.2cm}[f]\cdot a=[f*a],\hspace{.2cm} \alpha\cdot[f]=[f]\cdot\alpha=\varphi(\alpha)[f], \hspace{.4cm}(a\in\mathcal A, [f]\in\mathcal A/\mathbb{C}\varphi\circ\phi).$$
Setting $m\in\mathcal A^{**}$ with $\langle m,\varphi\circ\phi\rangle=1$ and defining a derivation $D:\mathcal A\longrightarrow\mathcal A^{**}$ via $D(a)=a\cdot m-\phi(a)\cdot m$, we find out $D(a)\in\{n\in\mathcal A^{**}:n(\varphi\circ\phi)=0\}=\{\mathbb{C}\varphi\circ\phi\}^\perp\cong(\mathcal A^*/\mathbb{C}\varphi\circ\phi)^*$. Due to the module approximate $(\phi,\varphi)$-amenability of $\mathcal A$, there exists a net $\{n_j\}\subset\{\mathbb{C}\varphi\circ\phi\}^\perp$ so that
$$D(a)=\lim_j(a\cdot n_j-n_j\cdot a)=\lim_j(a\cdot n_j-\phi(a)\cdot n_j)$$
Letting $m_j=m-n_j$, we have $\langle m_j,\varphi\circ\phi\rangle=\langle m,\varphi\circ\phi\rangle-\langle n_j,\varphi\circ\phi\rangle=1-0=1$ and
\begin{align*}
a\cdot m_j-\phi(a)\cdot m_j&=a\cdot m-a\cdot n_j-\phi(a)\cdot m-\phi(a)\cdot n_j\\
&=(a\cdot m-\phi(a)\cdot m)-(a\cdot n_j-\phi(a)\cdot n_j)\\
&\rightarrow 0
\end{align*}
for all $a\in \mathcal A$. Obviously, $\langle\alpha\cdot m_j-\varphi(\alpha)m_j,f\rangle\rightarrow 0$ for all $\alpha\in\mathfrak A$ and all $f\in\mathcal A^*$. Therefore, $\{m_j\}_j$ is a module approximate $(\phi,\varphi)$-mean on $\mathcal A^*$.

(ii)$\Rightarrow$(iii) It is clear.

(iii)$\Rightarrow$(i) Consider $\{m_j\}_j\subseteq\mathcal A^{**}$ such that $m_j(\varphi\circ\phi)\rightarrow 1$ and
$$a\cdot m_j-\phi(a)\cdot m_j\rightarrow 0,\hspace{.5cm}\alpha\cdot m_j-\varphi(\alpha)m_j\rightarrow 0$$
for all $a\in\mathcal A,\alpha\in\mathfrak A$. Suppose that $X$ is a $((\phi,\varphi),\mathcal A$-$\mathfrak A)$-bimodule and $D:\mathcal A\longrightarrow X^*$ is a bounded derivation. Put $D^{'}=D^{*}|_X:X\longrightarrow \mathcal A^*$ and $g_j:=(D^{'})^{*}(m_j)\in X^*$. Similar to the proof of implication (i)$\Rightarrow$(ii) from \cite[Theorem 2.1]{bod}, one can show that $D^{'}(x\cdot a)=D^{'}(x)\cdot a-\langle D(a),x\rangle\varphi\circ\phi$ for all $a\in\mathcal A$ and $x\in X$. This implies
\begin{align*}
\langle a\cdot g_j,x\rangle&=\langle g_j,x\cdot a\rangle=\langle (D^{'})^{*}(m_j),x\cdot a\rangle=\langle m_j,D^{'}(x\cdot a)\rangle\\
&=\langle m_j,D^{'}(x)\cdot a\rangle-\langle x,D(a)\rangle\langle m_j,\varphi\circ\phi\rangle\\
&=\langle a\cdot m_j,D^{'}(x)\rangle-\langle x,D(a)\rangle\langle m_j,\varphi\circ\phi\rangle.
\end{align*}
Hence
\begin{align*}
\|\langle a\cdot g_j,x\rangle-(\varphi\circ\phi)(a)\langle g_j,x\rangle+\langle D(a),x\rangle\|&\leq\|(\varphi\circ\phi)(a)\langle m_j,D^{'}(x)\rangle-\langle a\cdot m_j,D^{'}(x)\rangle\|\\
&+\|(\varphi\circ\phi)(a)\langle g_j,x\rangle-(\varphi\circ\phi)(a)\langle m_j,D^{'}(x)\|\\
&+\|\langle x,D(a)\rangle-\langle x,D(a)\rangle\langle m_j,\varphi\circ\phi\rangle\|\\
&\rightarrow 0
\end{align*}
The above relations show that
$$D(a)=\lim_j(\varphi\circ\phi)(a)g_j-a\cdot g_j=\lim_j(\phi(a)\cdot g_j-a\cdot g_j)\lim_j g_j\cdot a-a\cdot g_j=ad_{-g_{j}}$$
for all $a\in \mathcal A$. It follow that $\mathcal A$ is module approximately $(\phi,\varphi)$-amenable.

(i)$\Rightarrow$(iv) It is obvious.

(iv)$\Rightarrow$(ii)  Take $b\in\mathcal A$ with $(\varphi\circ\phi)(b)=1$. Clearly, $ab-ba\in K$ for all $a\in\mathcal A$. Thus, $D(a)=ab-\phi(a)\cdot b$ define a derivation from $\mathcal A$ into $K^{**}$. We have
$$\langle (\varphi\circ\phi)^{**},ab\rangle-\langle(\varphi\circ\phi)^{**},\phi(a)\cdot b\rangle=(\varphi\circ\phi)(a)(\varphi\circ\phi)(b)-(\varphi\circ\phi)(a)\langle(\varphi\circ\phi)^{**},b\rangle=0.$$
By assumption, there exists a net $\{n_j\}\subset K^{**}$ such that
$$D(a)=\lim_j a\cdot n_j-\phi(a)\cdot n_j\qquad (a\in\mathcal A).$$
Put $m_j=b-n_j$. Then, $\langle m_j,\varphi\circ\phi\rangle=\langle b,\varphi\circ\phi\rangle-\langle n_j,\varphi\circ\phi\rangle=1-0=1$ and
\begin{align*}
a\cdot m_j-\phi(a)\cdot m_j&=ab-a\cdot n_j-\phi(a)\cdot b+\phi(a)\cdot n_j\\
&=ab-\varphi(\phi(a))b-(a\cdot n_j-\varphi(\phi(a))n_j)\\
&\rightarrow 0
\end{align*}
Also, $\alpha\cdot m_j-\varphi(\alpha)m_j\rightarrow 0$ for all $\alpha\in \mathfrak A$. Therefore, $(m_j)$ is a module approximately $(\phi,\varphi)$-mean. The proof in the case of uniform module approximate $(\phi,\varphi)$- amenability and module $w^*$-approximate $(\phi,\varphi)$-amenability are similar.
\end{proof}

The proof idea of the following result is taken from the proof of \cite[Proposition 2.3]{ps}. We include its proof and stress on the details for the
sake of clarity.

\begin{proposition}\label{5th2}
Let $\mathcal A$ be a Banach $\mathfrak A$-module and $\varphi\in\Phi_{\mathfrak A}$ and $\phi\in\Omega_{\mathcal A}$. Then the following are equivalent:

\item[(i)] {There exists a net $\{u_j\}_j\subset\mathcal A^{**}$ such that $u_j(\varphi\circ\phi)\rightarrow 1$ and $a\cdot u_j-(\varphi\circ\phi)(a)u_j \rightarrow 0,\,\alpha\cdot u_j-\varphi(\alpha)u_j\rightarrow 0$ for all $a\in\mathcal A,\alpha\in\mathfrak A$;}
\item[(ii)]{ There exists a net $\{m_k\}_k\subset\mathcal A^{**}$ such that $m_k(\varphi\circ\phi)\rightarrow 1$ and $a\cdot m_k-(\varphi\circ\phi)(a)m_k\rightarrow 0,\,\alpha\cdot m_k-\varphi(\alpha)m_k\rightarrow 0$ for all $a\in\mathcal A,\alpha\in\mathfrak A$,
in the $w^*$-topology of $\mathcal A^{**}$;}

\item[(iii)] {There exists a net $\{n_l\}_l\subset\mathcal A$ such that $(\varphi\circ\phi)(n_l)\rightarrow 1$ and $a\cdot n_l-(\varphi\circ\phi)(a)n_l\rightarrow 0,\,\alpha\cdot n_l-\varphi(\alpha)n_l\rightarrow 0$ for all $a\in\mathcal A,\alpha\in\mathfrak A$.}
\end{proposition}

\begin{proof} Since the proof of implications (i)$\Rightarrow$(ii) and (iii)$\Rightarrow$(i) are clear, we only prove the implication (ii)$\Rightarrow$(iii). Take $\epsilon>0$ and finite sets $F\subset\mathcal A, \mathfrak H\subset\mathfrak A$ and $\Delta\subset\mathcal A^*$. Then, there exists $k$ such that $|\langle m_k,\varphi\circ\phi\rangle|>1-\epsilon$ and
$$|\langle a\cdot m_k-(\varphi\circ\phi)(a)m_k,f\rangle|\leq\frac{\epsilon}{3},\hspace{.3cm}|\langle\alpha\cdot m_k-\varphi(\alpha)m_k,f\rangle|<\frac{\epsilon}{3} \hspace{1cm}(a\in F,\,\alpha\in \mathfrak H,\,f\in\Delta)$$
By Goldstine's theorem, there exists $a_k\in\mathcal A$ such that
$$|\langle f,a_k\rangle-\langle m_k,f\rangle|<\frac{\epsilon}{3M}\hspace{1cm},f\in\Delta\cup(\Delta\cdot F)\cup\{\varphi\circ\phi\}\cup(\Delta\cdot\mathfrak H)$$
where $M=\sup\{|(\varphi\circ\phi)(a)|,|\varphi(\alpha)|;a\in F,\,\alpha\in \mathfrak H\}$. Thus, for any $f\in\Delta$ and $a\in F$, we have
\begin{align*}
|\langle f,a\cdot a_k-\varphi\circ\phi(a)a_k\rangle|&\leq|\langle a\cdot a_k-a\cdot m_k,f\rangle|+|\langle a\cdot m_k-\varphi\circ\phi(a)m_k,f\rangle|\\
&+|\langle \varphi\circ\phi(a)m_k-\varphi\circ\phi(a)a_k,f\rangle|\\
&\leq|\langle f\cdot a,a_k\rangle-\langle m_k,f\cdot a\rangle|+|\langle a\cdot m_k,f\rangle-\varphi\circ\phi(a)\langle m_k,f\rangle|\\
&+|\varphi\circ\phi(a)\langle m_k,f\rangle-\varphi\circ\phi(a)\langle f,a_k\rangle|\\
&\leq\frac{\epsilon}{3M}+\frac{\epsilon}{3}+|\varphi\circ\phi(a)|\frac{\epsilon}{3M}\\
&<\frac{\epsilon}{3}+\frac{\epsilon}{3}+\frac{\epsilon}{3}=\epsilon.
\end{align*}
Similarly, one can show that  $|\langle f,\alpha\cdot a_k-\varphi(\alpha)a_k\rangle|<\epsilon$ for all $f\in\Delta$ and $\alpha\in \mathfrak H$. So, there exists a net $\{b_l\}_l\subset\mathcal A$ such that for each $a\in\mathcal A$, $\varphi\circ\phi(b_l)\rightarrow 1$ and $$a\cdot b_l-\varphi\circ\phi(a)b_l\rightarrow 0,\hspace{.3cm}\alpha\cdot b_l-\varphi(\alpha)b_l\rightarrow 0$$
weakly in $\mathcal A$. Lastly, for each finite sets $F=\{a_1,a_2,...,a_n\},\mathfrak H=\{\alpha_1,...,\alpha_m\}$ of $\mathcal A$ and $\mathfrak A$, respectively, consider the set
$$C=\{((a_ib_l-\varphi\circ\phi(a_i)b_l)_{i=1}^n,(\alpha_j\cdot b_l-\varphi(\alpha_j)b_l)_{j=1}^m,\varphi\circ\phi(b_l)); b\in\mathcal A\}.$$
Hence, in the Banach space $\mathcal A^{n+m}\times\mathbb{C}$, $(\overbrace{0,0,...,0}^{n+m-times},1)$ is in the weak cluster of the convex hull of $C$. Using Mazur's theorem,  for each $\epsilon>0$, there exists $V_{\epsilon, F, \mathfrak H}\in\text{co}\{b_l\}$ such that $\|aV_{\epsilon, F, \mathfrak H}-\varphi\circ\phi(a)V_{\epsilon, F, \mathfrak H}\|<\epsilon$, $\|\alpha\cdot V_{\epsilon, F, \mathfrak H}-\phi(\alpha)V_{\epsilon, F, \mathfrak H}\|<\epsilon$ and $|\varphi\circ\phi(V_{\epsilon, F, \mathfrak H})-1|<\epsilon$ for all $a\in F, \alpha\in \mathfrak H$. Therefore, there exists a net $\{n_l\}_l\subseteq \mathcal A$ such that $(\varphi\circ\phi)(n_l)\rightarrow 1$ and for all $a\in\mathcal A,\alpha\in\mathfrak A$
$an_l-\varphi\circ\phi(a)n_l\rightarrow 0, \alpha\cdot n_l-\varphi(\alpha)n_l\rightarrow 0$. This completes the proof.
\end{proof}


\begin{proposition}
Let $\mathcal A$ and $\mathcal B$ be $\mathfrak A$-modules, and $\theta:\mathcal A\longrightarrow\mathcal B$ be continues $\mathfrak A$-module epimorphism. Then module approximately $(\phi\circ\theta,\varphi)$-amenability of $\mathcal A$ implies module approximately $(\phi,\varphi)$-amenability of $\mathcal B$.
\end{proposition}

\begin{proof}
By hypothesis, there exists a module approximate $(\phi\circ\theta,\varphi)$-mean $\{m_j\}\subseteq\mathcal A^{**}$ such that $m_j(\varphi\circ\phi\circ\theta)=1$ and $a\cdot m_j-(\phi\circ\theta)(a)\cdot m_j\rightarrow 0,\,\alpha\cdot m_j-\varphi(\alpha)m_j\rightarrow 0$ for all $a\in\mathcal A,\alpha\in\mathfrak A$. For any $j$, define $n_j\in\mathcal B^{**}$ via $n_j(g)=m_j(g\circ\theta)$ where $g\in\mathcal B^*$. Then, $n_j(\varphi\circ\phi)=m_j(\varphi\circ\phi\circ\theta)=1$. Similar to the proof of \cite[Proposition 2.3]{bod}, we can show that $(g\cdot\alpha)\circ\theta=(g\circ\theta)\cdot\alpha$ and $(g\cdot\theta(a))\circ\theta=(g\circ\theta)\cdot a$ for all $a\in\mathcal A,\alpha\in\mathfrak A$ and $g\in\mathcal B^*$. So we have
\begin{align*}
\langle n_j,g\cdot\alpha\rangle-\varphi(\alpha)\langle n_j,g\rangle &=\langle m_j,(g\cdot\alpha)\circ\theta\rangle-\varphi(\alpha)\langle m_j,g\circ\theta\rangle\\
&=\langle m_j,(g\circ\theta)\cdot\alpha\rangle-\varphi(\alpha)\langle m_j,g\circ\theta\rangle\\
&=\langle\alpha\cdot m_j,g\circ\theta\rangle-\varphi(\alpha)\langle,g\circ\theta\rangle\\
&=\langle \alpha\cdot m_j-\varphi(\alpha)m_j,g\circ\theta\rangle\\
&\rightarrow 0
\end{align*}
for all $\alpha\in\mathfrak A$ and $g\in\mathcal B^*$. Hence, $\alpha\cdot m_j-\varphi(\alpha)m_j\rightarrow0$for all $\alpha\in\mathfrak A$ and $g\in\mathcal B^*$. On the other hand,
\begin{align*}
\langle \theta(a)\cdot n_j-(\varphi\circ\phi)(\theta(a))n_j,g\rangle&=\langle n_j,g\cdot\theta(a)\rangle-(\varphi\circ\phi)(\theta(a))\langle n_j,g\rangle\\
&=\langle m_j,(g\cdot\theta(a))\circ\theta\rangle-(\varphi\circ\phi)(\theta(a))\langle m_j,g\circ\theta\rangle\\
&=\langle m_j,(g\circ\theta)\cdot a\rangle-(\varphi\circ\phi)(\theta(a))\langle m_j,g\circ\theta\rangle\\
&=\langle a\cdot m_j,g\circ\theta\rangle-(\varphi\circ\phi\circ\theta)(a)\langle m_j,g\circ\theta\rangle\\
&=\langle a\cdot m_j-(\varphi\circ\phi\circ\theta)(a)m_j,g\circ\theta\rangle\\
&\rightarrow 0
\end{align*}
 for all $a\in\mathcal A$ and $g\in\mathcal B^*$. Therefore, $(n_j)$ is module approximate $(\varphi\circ\phi)$-mean for $\mathcal B$.
\end{proof}
\begin{corollary}\label{coo}
Let $\mathcal A$ and $\mathcal B$ be Banach $\mathfrak A$-modules and $\theta:\mathcal A\longrightarrow\mathcal B$ be a continues $\mathfrak A$-bimodule epimorphism. Then the module approximately character amenability of $\mathcal A$ implies the module approximately character amenability of $\mathcal B$. In particular, if $\mathcal A$ is approximately character module amenable, then so is $\mathcal A/J$.
\end{corollary}


Module approximate character amenability of $\ell^1(S)$ and its second dual  for
an inverse semigroup $S$ is characterized in the following result.

\begin{theorem}\label{tm} Let $S$ be an inverse semigroup with the set of idempotents $E$. Then

\begin{enumerate}
\item[(i)] {$\ell^1(S)$ is $\ell^1(E)$-module approximately
character amenable if and only if $S$ is amenable;}
\item[(ii)] {$\ell
^{1}(S)^{**}$ is $\ell^1(E)$-module approximately character amenable if and only if
$G_S$ is finite.}
\end{enumerate}
\end{theorem}
\begin{proof} (i)
Let $\ell^1(S)$ be module approximately charater amenable. It follows from \cite[Remark 2.5]{bod} that $\ell^1(S)/J\cong\ell^1(G_S)$ is  approximately
character amenable. Since $G_S$ is a discrete group,
by \cite[Theorem 7.1]{ps}, it is amenable and so $S$ is amenable.

Conversely, if $S$ is amenable, then by \cite[Theorem 3.1]{am1} (or \cite[Theorem 2.9]{pou}) we
conclude that $\ell^1(S)$ is module amenable and hence module
approximately character amenable.

(ii)  If $G_S$ is finite, then $\ell ^1(S)^{**}$ module amenable by \cite[Theorem 3.4]{abe} (see also \cite[Theorem 2.11]{pou}). Thus $\ell ^1(S)^{**}$ is module approximately character amenable.

Conversely, let $\mathcal N$ be the closed ideal of ${\mathcal
A^{**}}$ and $\lambda$ be the map before Theorem \ref{thj}. By \cite[Remark 2.5]{bod}, $\ell ^1(S)^{**}/\mathcal N$ is approximately character amenable. Since $\lambda$ is a continuous epimorphism, by Corollary \ref{coo}, $\ell ^{1}(S)^{**}/\mathcal J^{\perp \perp}\cong \ell^1(G_S)^{**}$ is approximately character amenable. Now, by applying Theorem 7.4 from \cite{ps}, we see that $G_S$ is finite.
\end{proof}


Let $\mathcal A$ be a Banach algebra and a Banach $\mathfrak A$-bimodule with compatible actions. Suppose $\phi\in\Omega_{\mathcal A}$ and $\widetilde{\varphi}$ is the extention of $\varphi$ on $\mathfrak A^\sharp=\mathfrak A\oplus\mathbb{C}$ defined by $\widetilde{\varphi}(\alpha,\lambda)=\varphi(\alpha)+\lambda$, for all $\alpha\in\mathfrak A,\lambda\in\mathbb{C}$. If $u=(\alpha,\lambda)\in\mathfrak A^\sharp$, it is easy to check that $$\phi(a\cdot u)=\phi(u\cdot a)=\widetilde{\varphi}(u)\phi(a)\hspace{.5cm}(a\in\mathcal A).$$
Define $\widetilde{\phi}:\mathcal A^\sharp=\mathcal A\oplus\mathfrak A^\sharp\longrightarrow\mathfrak A^\sharp$ by $\widetilde{\phi}(a,u)=(\phi(a),\widetilde{\varphi}(u))$ ($\mathcal A^\sharp$ is the module unitization of $\mathcal A$). We can eaily verify that $\widetilde{\phi}$ is multiplicative and
$$\widetilde{\phi}(u\cdot (a,v))=\widetilde{\phi}((a,v)\cdot u)=\widetilde{\varphi}(u)\widetilde{\phi}(a,v).$$
Therefore $\widetilde{\phi}$ is an extension of $\phi$ such that $\widetilde{\phi}(0,u)=\widetilde{\varphi}(u)$ is the extension $h_0=\tilde{0}$ of the zero function (for more details see \cite{bod}).

\begin{proposition}\label{5pro1}
Let $\mathcal A$ be a Banach $\mathfrak A$-module Banach without identity and let $\varphi\in\Phi_{\mathfrak A}$ and $\phi\in\Omega_{\mathcal A}$, then $\mathcal A$ is approximately module $(\phi,\varphi)$-amenable if and only if $\mathcal A^\sharp$
is approximately module $(\widetilde{\phi},\widetilde{\varphi})$-amenable.
\end{proposition}
\begin{proof}
 If $\varphi=0$, then $\varphi\circ\phi=0$ and $\mathcal A^\sharp=\mathbb{C}e\oplus\mathcal A$ where $e =(0,1)$ is
the identity of $\mathfrak A^\#$. So $(\widetilde{\varphi}\circ\widetilde{\phi})(\mathcal A)=0$ and $\widetilde{\varphi}\circ\widetilde{\phi}(e)=1$. Now, similar to the proof of \cite[Proposition 2.8]{ps}, we can prove that $\mathcal A$ has a right approximate identity, say $\{b_j\}$. Put $m_j=e-b_j$. Then $(\widetilde{\varphi}\circ\widetilde{\phi})(m_\alpha)=(\widetilde{\varphi}\circ\widetilde{\phi})(e)-(\widetilde{\varphi}\circ\widetilde{\phi})(b_\alpha)=1$ and
$$\|am_j-(\widetilde{\varphi}\circ\widetilde{\phi})(a)m_j\|=\|am_j\|=\|ae-ab_j\|\rightarrow 0\hspace{1cm}(a\in\mathcal A^\sharp).$$
Also, for each $u\in\mathcal A^\sharp$, we have
$$\|u\cdot m_j-\widetilde{\varphi}(u)m_\alpha\|\rightarrow 0.$$
By Propositions \ref{5th1} and \ref{5th2}, $\mathcal A^\sharp$ is approximately module $(\widetilde{\phi},\widetilde{\varphi})$-amenable.
If $\varphi\in\Phi_{\mathfrak A}$, the process of the proof of \cite[Proposition 2.7]{bod} can be repeated to get the result by substituting the net of module means instead of the module mean.
\end{proof}


\begin{proposition}
Let $\mathcal A$ be a Banach $\mathfrak A$-module and let $\varphi\in\Phi_{\mathfrak A}$ and $\phi\in\Omega_{\mathcal A}$, then $\mathcal A$ is module approximately $(\phi,\varphi)$-amenable if and only if $\mathcal A^{**}$ is module approximately $(\phi^{**},\varphi^{**})$-amenable.
\end{proposition}
\begin{proof}
Assume that $\mathcal A^{**}$ is module approximately $(\phi^{**},\varphi^{**})$-amenable and $\varphi=0$. Since $\mathcal A^{**}$ has a right approximate identity (see the proof of Proposition \ref{5pro1}), $\mathcal A$ has also a right approximate identity, and thus $\mathcal A$ is approximate module $(\phi,\varphi)$-amenable. If $\varphi\neq0$, let $\{m_\alpha\}\subset\mathcal A^{****}$ be an approximate module $(\phi^{**},\varphi^{**})$-mean for $\mathcal A^{**}$, then by Theorem \ref{5th1}, $\{m_\alpha|_{\mathcal A^{*}}\}$ forms an approximate module $(\phi,\varphi)$-mean for $\mathcal A$ and thus $\mathcal A$ is approximate module $(\phi,\varphi)$-amenable.

Conversely, let $\mathcal A$ be approximate module  $(\phi,\varphi)$-amenable and $\{m_l\}$ be a approximate module $(\phi,\varphi)$-mean in $\mathcal A^{**}$. For each $F\in\mathcal A^{**}$ and $\psi\in\mathcal A^{***}$, take bounded nets $(a_j)\in\mathcal A$ and $(f_k)\in\mathcal A^*$ with $\hat{a}_j^{**}\rightarrow F$ and $\hat{f}_k^{**}\rightarrow\psi$ in the $w^*$-topology. For any $l$, we consider $m_l$ as an element $m^{**}$ of $\mathcal A^{***}$. Hence $m_l^{**}(\varphi^{**}\circ\phi^{**})=m_l(\varphi\circ\phi)=1$ and
\begin{align*}
\lim_l\langle F\cdot m_l^{**}-\varphi^{**}\circ\phi^{**}(F)m_l^{**},\psi\rangle&=\lim_l\langle F\cdot m_l^{**},\psi\rangle-\lim_l\langle (\varphi^{**}\circ\phi^{**})(F)m_l^{**},\psi\rangle\\
&=\lim_l\langle m_l^{**},\psi\cdot F\rangle-\varphi^{**}\circ\phi^{**}(F)\lim_l\langle m_l^{**},\psi\rangle\\
&=\lim_l\langle \psi,F\square m_l\rangle-\varphi^{**}\circ\phi^{**}(F)\lim_l\langle m_l^{**},\psi\rangle\\
&=\lim_l\lim_k\langle F\square m_l,f_k\rangle-\varphi^{**}\circ\phi^{**}(F)\lim_l\langle m_l^{**},\psi\rangle\\
&=\lim_l\lim_k\langle F,m_l\cdot f_k\rangle-\varphi^{**}\circ\phi^{**}(F)\lim_l\langle m_l^{**},\psi\rangle\\
&=\lim_l\lim_k\lim_j\langle m_l\cdot f_k,a_j\rangle-\varphi^{**}\circ\phi^{**}(F)\lim_l\langle m_l^{**},\psi\rangle\\
&=\lim_l\lim_k\lim_j\langle m_l,f_k\cdot a_j\rangle-\varphi^{**}\circ\phi^{**}(F)\lim_l\langle m_l^{**},\psi\rangle\\
&=\lim_i\lim_k\lim_j\varphi\circ\phi(a_j)\langle m_l,f_k\rangle-\varphi^{**}\circ\phi^{**}(F)\lim_l\langle m_l^{**},\psi\rangle\\
&=\lim_l\lim_j\varphi\circ\phi(a_j)\lim_k\langle m_l,f_k-\varphi^{**}\circ\phi^{**}(F)\lim_l\langle m_l^{**},\psi\rangle\\
&=\lim_l\varphi^{**}\circ\phi^{**}(F)\langle m_l^{**},\psi\rangle-\lim_l\varphi^{**}\circ\phi^{**}(F)\langle m_l^{**},\psi\rangle\\
&=0
\end{align*}
Now, for any $\mathfrak F\in\mathfrak A$ and $\psi\in\mathcal A^{***}$, similar to the above computation, we can show $\lim_j\langle\mathfrak F\cdot m_l^{**}-\varphi^{**}(\mathfrak F)m_l^{**},\psi\rangle=0$. This means that $\mathcal A^{**}$ is module approximately  $(\phi^{**},\varphi^{**})$-amenable.
\end{proof}
\begin{corollary}
Let $\mathcal A$ be a Banach $\mathfrak A$-module. If $\mathcal A^{**}$ is module approximately character amenable, then so is $\mathcal A$.
\end{corollary}

Let $\mathcal A$ and $\mathcal B$ be Banach algebras and $\mathcal
A\widehat{\otimes}\mathcal B$ be the projective tensor product of
$\mathcal A$ and $\mathcal B$. Then
$\mathcal A\widehat{\otimes}\mathcal B$ is a Banach $\mathfrak
A\widehat{\otimes}\mathfrak A$-module with the following actions:
$$(\alpha\otimes\beta)\cdot(a\otimes b)=(\alpha\cdot a)\otimes (\beta\cdot
b)\quad(a\in\mathcal A, b\in \mathcal B, \alpha,\beta\in\mathfrak
A),$$ and similarly for the right action. For
$\varphi\in\Omega_{\mathcal A}$ and $\psi\in\Omega_{\mathcal B}$,
consider $\varphi\otimes\psi$ by $(\varphi\otimes\psi)(a\otimes
b)=\varphi(a)\psi(b)$ for all $a\in \mathcal A$ and $b\in \mathcal
B$. Clearly, $\varphi\otimes\psi\in\Omega_{\mathcal
A\widehat{\otimes}\mathcal B}$. Also if $\varphi_1,
\varphi_2\in\Phi_{\mathfrak A} \cup\{0\}$, then
$\varphi_1\otimes\varphi_2\in\Phi_{\mathfrak
A\widehat{\otimes}\mathfrak A} \cup\{0\}$, and if $\bar{\varphi}\in\Phi_{\mathfrak
A\widehat{\otimes}\mathfrak A}$, then $\bar{\varphi}=\varphi_1\otimes\varphi_2$ where $\varphi_1,
\varphi_2\in\Phi_{\mathfrak A} $.

\begin{proposition}
Let $\mathcal A$ and $\mathcal B$ be Banach $\mathfrak A$-modules, and let $\phi\in\Omega_{\mathcal A},\psi\in\Omega_{\mathcal B}$ and $\varphi_1,\varphi_2\in\Phi_{\mathfrak A}$.

\item[(i)] If $\mathcal A\widehat{\otimes}\mathcal B$ is module approximately $(\phi\otimes\psi,\varphi_1\otimes\varphi_2)$-amenable (as
$\mathfrak A\widehat{\otimes}\mathfrak A$-module), then $\mathcal A$ is module approximately $(\phi,\varphi_1)$-amenable and $\mathcal B$ is module approximately $(\psi,\varphi_2)$-amenable.
\item[(ii)] If $\mathcal A$ is module $(\phi,\varphi_1)$-amenable and $\mathcal B$ is module approximately $(\psi,\varphi_2)$-amenable, then $\mathcal A\widehat{\otimes}\mathcal B$ is module approximately $(\phi\otimes\psi,\varphi_1\otimes\varphi_2)$-amenable.
\end{proposition}
\begin{proof}
Assume that $(m_j)\subset(\mathcal A\widehat{\otimes}\mathcal B)^{**}$ is a module approximate $(\phi\otimes\psi,\varphi_1\otimes\varphi_2)$-mean for $\mathcal A\widehat{\otimes}\mathcal B$. Then, $m_j((\varphi_1\otimes\varphi_2)\circ(\phi\otimes\psi))=1$ for all $\alpha$, Also,
$$\|(a\otimes b)\cdot m_j-m_j((\varphi_1\otimes\varphi_2)\circ(\phi\otimes\psi))(a\otimes b)m_j\|\rightarrow 0,\hspace{.2cm}\|(\alpha_1\otimes\alpha_2)\cdot m_j-(\varphi_1\otimes\varphi_2)(\alpha_1\otimes\alpha_2)m_j\rightarrow 0$$
for all $a\in\mathcal A,b\in\mathcal B,\alpha_1,\alpha_2\in\mathfrak A,f\in\mathcal A^*$. We choose $a_0\in\mathcal A$ and $b_0\in\mathcal B$ such that $(\varphi_1\circ\phi)(a_0)=(\varphi_2\circ\psi)(b_0)=1$. Define $\{\overline{m}_j\}\subseteq\mathcal A^{**}$ by $\overline{m}_j(f)=m(f\otimes(\varphi_2\circ\psi))\,(f\in\mathcal A^*)$. For each $a\in\mathcal A$ and $f\in\mathcal A^*$, we have
\begin{align*}
\lim_j\langle a\cdot\overline{m}_j-(\varphi_1\circ\phi)(a)\overline{m}_j,f\rangle&=((\varphi_1\circ\phi)\otimes(\varphi_2\circ\psi))(a_0\otimes b_0)\lim_j\langle m_j,(f\cdot a)\otimes(\varphi_2\circ\psi)\rangle\\
&=\lim_j\langle m_j,(f\cdot a)\otimes(\varphi_2\circ\psi)\cdot(a_0\otimes b_0)\rangle\\
&=\lim_j\langle m_j,f\cdot(aa_0)\otimes(\varphi_2\circ\psi)\cdot b_0\rangle\\
&=\lim_j\langle m_j,f\otimes(\varphi_2\circ\psi)\cdot(aa_0\otimes b_0\rangle)\\
&=((\varphi_1\circ\phi)\otimes(\varphi_2\circ\psi))(aa_0\otimes b_0)\lim_j\langle m_j,f\otimes(\varphi_2\circ\psi)\\
&=(\varphi_1\circ\phi)(a)(\varphi_1\circ\phi)(a_0)(\varphi_2\circ\psi)(b_0)\lim_j\langle m_j,f\otimes(\varphi_2\circ\psi)\rangle\\
&=(\varphi_1\circ\phi)(a)\lim_j\overline{m}_j(f)
\end{align*}
We now take $\gamma_1,\gamma_2\in\mathfrak A$ such that $\varphi_1(\gamma_1)=\varphi_2(\gamma_2)=1$. For each $\alpha\in\mathfrak A$ and $f\in\mathcal A^*$, we have
\begin{align*}
\lim_j\langle\alpha\cdot\overline{m}_j-\varphi_1(\alpha)\overline{m}_j,f\rangle&=\lim_j\langle \overline{m}_j,f\cdot\alpha\rangle-\varphi_1(\alpha)\lim_j\langle\overline{m}_j,f\rangle\\
&=\lim_j(\varphi_1\otimes\varphi_2)(\gamma_1\otimes\gamma_2)\langle m_j,(f\cdot\alpha)\otimes(\varphi_2\circ\psi)\rangle-\varphi_1(\alpha)\lim_j\langle m_j,f\otimes(\varphi\circ\psi)\rangle\\
&=\lim_j\langle m_j,((f\cdot\alpha)\otimes(\varphi_2\circ\psi))\cdot(\gamma_1\otimes \gamma_2)\rangle-\varphi_1(\alpha)\lim_j\langle m_j,f\otimes(\varphi\circ\psi)\rangle\\
&=\lim_j\langle m_j,(f\cdot(\alpha\gamma_1))\otimes((\varphi_2\circ\psi)\cdot\gamma_2)\rangle-\varphi_1(\alpha)\lim_\alpha\langle m_j,f\otimes(\varphi\circ\psi)\rangle\\
&=\lim_j\langle m_j,(f\otimes(\varphi_2\circ\psi))\cdot(\alpha\gamma_1\otimes\gamma_2)\rangle-\varphi_1(\alpha)\lim_j\langle m_j,f\otimes(\varphi\circ\psi)\rangle\\
&=(\varphi_1\otimes\varphi_2)(\alpha\gamma_1\otimes\gamma_2)\lim_j\langle m_j,f\otimes(\varphi_2\circ\psi)\rangle-\varphi_1(\alpha)\lim_j\langle m_j,f\otimes(\varphi\circ\psi)\rangle\\
&=\varphi_1(\alpha)\varphi_1(\gamma_1)\varphi_2(\gamma_2)\lim_j\langle m_j,f\otimes(\varphi_2\circ\psi)\rangle-\varphi_1(\alpha)\lim_j\langle m_j,f\otimes(\varphi\circ\psi)\rangle\\
&=\varphi_1(\alpha)\lim_j\langle m_j,f\otimes(\varphi_2\circ\psi)\rangle-\varphi_1(\alpha)\lim_j\langle m_j,f\otimes(\varphi\circ\psi)\rangle\\
&=0.
\end{align*}
Thus, we get $\|\alpha\cdot\overline{m}_j-\varphi_(\alpha)\overline{m}_j\|\rightarrow0$. Also, for any $j$, we have $\overline{m}_j(\varphi_1\circ\phi)=m_j((\varphi_1\otimes\varphi_2)\circ(\phi\otimes\psi))=1$. Therefore, $\mathcal A$ is module approximately $(\phi,\varphi_1)$-amenable. Similarly $\mathcal B$ is module approximately $(\psi,\varphi_2)$-amenable. The proof of part (ii) is  similar to the proof of necessary implication of \cite[Theorem 2.8]{bod} and so we omit it.
\end{proof}

\section*{Acknowledgments}
The authors express their sincere thanks to the referee for his/her careful and detailed reading of the manuscript and very
helpful suggestions. The second and third authors would like to thank the University of Isfahan for its financial support.


\begin{thebibliography}{33}
\bibitem{am1} M. Amini,  Module amenability for semigroup algebras,
 Semigroup Forum. 69 (2004), 243--254.

\bibitem{abe} M. Amini, A. Bodaghi and  D. Ebrahimi Bagha,  Module amenability of the second dual and
module topological center of semigroup algebras, Semigroup
Forum. 80 (2010), 302--312.

\bibitem{abeh}M. Amini, A. Bodaghi, M. Ettefagh and K. Haghnejad Azar,
Arens regularity and module Arens regularity of module actions, U.P.B. Sci. Bull., Series A. 74, Iss. 2, (2012), 75--86.

\bibitem{ame} M. Amini and  D. Ebrahimi Bagha,   Weak module amenability for semigroup algebras,
Semigroup Forum. 71 (2005), 18--26.

\bibitem{bodaghi} A. Bodaghi,  Generalized notion of weak module amenability,
Hacettepe J. Math. Stat. 43(1) (2014), 85--95.

\bibitem{bo3} A. Bodaghi,  Module contractibility for semigroup algebras,
Math. Sci. J. 7, No. 2 (2012), 5--18.

\bibitem{boda} A. Bodaghi, Module $(\varphi,\psi)$-amenability of Banach algebras,
 Arch. Math (Brno). 46 (2010), no. 4, 227--235.
\bibitem{bod} A. Bodaghi and M. Amini,  Module character amenability of Banach algebras,
 Arch. Math (Basel). 99 (2012), 353--365.

 \bibitem{bab} A. Bodaghi, M. Amini and R. Babaee,  Module derivations into iterated duals of Banach algebras,
 Proc. Romanian .Aca., Series A. 12 (2011), 277--284.

\bibitem{dal} H. G. Dales, A. T.-M. Lau and D. Strauss, Banach algebras on semigroups and on their compactifications, Memoirs Amer. Math. Soc., 210 (2010), 165 pp.

\bibitem{dun} J. Duncan, I. Namioka, Amenability of inverse semigroups and their semigroup algebras,
 Proc. Roy. Soc. Edinburgh. 80A (1988), 309--321.

\bibitem{ef}E. Essmaili and M. Filali,
$\phi$-Amenability and character amenability of some classes of Banach algebras, Houston J. Math. 39, No. 2 (2013), 515--529.

\bibitem{gh1} F. Ghahramani and R. J. Loy,  Generalized notions of amenability, J. Funct. Anal.,  208 (2004), 229--260.

\bibitem{gh2} F. Ghahramani, R. J. Loy and Y. Zhang, Generalized notions of amenability II, J. Funct. Anal.,  254 (2008), 1776–-1810.


\bibitem {hum}Z. Hu, M. S. Monfared, and T. Traynor,  On character amenable Banach
algebras,  Studia Math. 193 (2009), 53--78.

\bibitem {joh} B. E. Johnson, Cohomology in Banach algebras,
Memoirs Amer. Math. Soc. 127, Providence, 1972.

\bibitem{kan1} E. Kaniuth, A. T. Lau, and J. Pym,
On $\varphi$-amenability of Banach algebras,  Math. Proc. Camb.
Soc. 144 (2008), 85--96.

\bibitem {kan2}E. Kaniuth, A.T. Lau, J. Pym, On character amenability of Banach
algebras,  J. Math. Anal. Appl. 344 (2008), 942--955.

\bibitem{lva} M. Lashkarizadeh Bami,·M. Valaei·and M. Amini,   Super module amenability of inverse semigroup
algebras, Semigroup Forum. 86 (2013), 279--288.

\bibitem{mon} M. S. Monfared, Character amenability of Banach algebras,
 Math. Proc. Camb. Soc. 144 (2008), 697--706.

\bibitem{mn} W.D. Munn, A class of irreducible matrix
representations of an arbitrary inverse semigroup,  Proc.
Glasgow Math. Assoc. 5 (1961), 41--48.

\bibitem{pou} H. Pourmahmood-Aghababa, (Super) Module amenability, module topological
center and semigroup algebras,  Semigroup Forum. 81
(2010), 344--356.

\bibitem{pou2} H. Pourmahmood-Aghababa, A note on two equivalence relations on inverse
semigroups, Semigroup Forum. 84 (2012), 200--202.

\bibitem{pa} H. Pourmahmood-Aghababa and A. Bodaghi,  Module approximate amenability of Banach algebras, Bull. Iran. Math. Soc.  39, No. 6 (2013), 1137--1158.

\bibitem{ps} H. Pourmahmood-Aghababa, L. Y. Shi and Y. J. Wu,  Generalized notions of character amenability, Acta Math. Sinica, English
Series, 29, Iss. 7 (2013), 1329--1350.
\bibitem{ra1} R. Rezavand, M. Amini, M. H. Sattari, and D. Ebrahimi Bagha,
Module Arens regularity for semigroup algebras,  Semigroup Forum.  77 (2008), 300--305.

\bibitem{rie} M. A. Rieffel, Induced Banach representations of Banach algebras and locally
 compact groups,  J. Funct. Anal. 1 (1967), 443--491.

 \bibitem{sc} H. H. Schaefer,  Topological Vector Space, Springer-Verlag, 1979.
 
 \bibitem{s} P. $\check{S}$emrl, Additive derivations of some operator algebras, Illinois J. Math. 35 (1991), 234--240.

\end{thebibliography}
\end{document}